\documentclass[
	english
]{scrartcl}

\usepackage[T1]{fontenc}
\usepackage[utf8]{inputenc}

\usepackage{amsmath,amsfonts,amsthm,amssymb}

\usepackage[colorlinks,linkcolor=blue,citecolor=blue]{hyperref}

\usepackage{preamble}
\usepackage{booktabs}

\usepackage{enumitem}
\setenumerate{label=(\roman*)}

\usepackage[ruled,vlined,linesnumbered,algo2e,algosection]{algorithm2e}

\numberwithin{equation}{section} 

\usepackage{lmodern}
\usepackage[final]{microtype}

\usepackage[normalem]{ulem}

\usepackage{mathtools}

\usepackage[capitalise,nameinlink]{cleveref}

\newif\ifbiber
\bibertrue

\ifbiber
\usepackage[
	style=authoryear-comp,
	sorting=nyt,
	url=true,
	doi=true,                
	isbn=false,              
	eprint=true,
	giveninits=true,         
	uniquename=init,         
	maxbibnames=99,          
	maxcitenames=2,
	uniquelist=minyear,
	dashed=false,            
	sortcites=false,          
	backend=biber,           
	safeinputenc,            
	date=year,               
	useprefix=true,          
]{biblatex}

\DeclareCiteCommand{\cite}{%
	\ifbibmacroundef{cite:init}{}{\usebibmacro{cite:init}}\usebibmacro{prenote}%
}{%
	\usebibmacro{citeindex}%
	\printtext[bibhyperref]{\usebibmacro{cite}}%
}{%
	\ifbibmacroundef{cite:init}{\multicitedelim}{}%
}{%
	\usebibmacro{postnote}%
}%
\DeclareCiteCommand{\parencite}[\mkbibbrackets]{%
	\ifbibmacroundef{cite:init}{}{\usebibmacro{cite:init}}\usebibmacro{prenote}%
}{%
	\usebibmacro{citeindex}%
	\printtext[bibhyperref]{\usebibmacro{cite}}%
}{%
	\ifbibmacroundef{cite:init}{\multicitedelim}{}%
}{%
	\usebibmacro{postnote}%
}%

\let\cite\parencite

\DeclareNameAlias{sortname}{family-given}
\DeclareNameAlias{default}{family-given}

\else

\usepackage{doi,natbib}

\fi

\providecommand\given{\nonscript\;\delimsize|\nonscript\;\mathopen{}}
\DeclarePairedDelimiterX\set[1]\{\}{#1}
\DeclarePairedDelimiterX\seq[1](){#1}

\DeclarePairedDelimiterX\dual[2]{\langle}{\rangle}{#1,#2}
\DeclarePairedDelimiterX\innerprod[2](){#1,#2}

\DeclarePairedDelimiter\abs{\lvert}{\rvert}
\DeclarePairedDelimiter\norm{\lVert}{\rVert}

\DeclarePairedDelimiter\parens()
\DeclarePairedDelimiter\bracks[]

\newcommand\N{\mathbb{N}}
\newcommand\R{\mathbb{R}}

\renewcommand\d{\mathop{}\!\mathrm{d}}

\newcommand{\embeds}{\hookrightarrow}

\newcommand{\mto}{\rightrightarrows}

\newcommand{\dualspace}{^\star}

\newcommand\HH{\mathcal{H}}

\newcommand\LL{\mathcal{L}}

\newcommand\Cjet{C_{\mathrm{jet}}}
\newcommand\Cb{C_{\mathrm{b}}}

\DeclareMathAlphabet{\mathpzc}{OT1}{pzc}{m}{it}

\newcommand\sign{\operatorname{sign}}
\newcommand\Sign{\operatorname{Sign}}

\newcommand\supp{\operatorname{supp}}

\newcommand\dist{\operatorname{dist}}
\newcommand\diam{\operatorname{diam}}
\newcommand\id{\operatorname{id}}

\let\epsilon\varepsilon

\newcommand\orcid[1]{%
	\hspace{.25em}%
	\href{http://orcid.org/#1}{%
		\protect\includegraphics[height=1em]{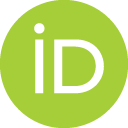}%
	}%
	\hspace{.25em}%
}

\newtheorem{theorem}{Theorem}[section]
\newtheorem{lemma}[theorem]{Lemma}
\newtheorem{proposition}[theorem]{Proposition}
\newtheorem{assumption}[theorem]{Assumption}
\newtheorem{corollary}[theorem]{Corollary}
\newtheorem{remark}[theorem]{Remark}
\newtheorem{definition}[theorem]{Definition}
\newtheorem{example}[theorem]{Example}
\crefname{assumption}{Assumption}{Assumptions}

\def\cleverreffix#1{\AddToHook{env/#1/begin}{\crefalias{theorem}{#1}}}
\cleverreffix{lemma}
\cleverreffix{proposition}
\cleverreffix{assumption}
\cleverreffix{corollary}
\cleverreffix{remark}
\cleverreffix{definition}
\cleverreffix{example}

\begin{document}
\title{Continuous differentiability of the signum function and Newton's method for bang-bang control}

\author{%
	Daniel Wachsmuth%
	\footnote{%
		Universität Würzburg,
		Institut für Mathematik,
		97074 Würzburg, Germany
	}
	\orcid{0000-0001-7828-5614}
	\and
	Gerd Wachsmuth%
	\footnote{%
		Brandenburgische Technische Universität Cottbus-Senftenberg,
		Institute of Mathematics,
		03046 Cott\-bus, Germany
	}
	\orcid{0000-0002-3098-1503}%
}
\publishers{}

\maketitle

\begin{abstract}
We investigate bang-bang control problems and the possibility to apply
Newton's method to solve such kind of problems numerically.
To this end, we show that the signum function is Fréchet differentiable between appropriate function spaces.
Numerical experiments show the applicability of the resulting method.
\end{abstract}


\begin{msc}
	\mscLink{49M15}, 
	\mscLink{49K30}, 
	\mscLink{49M41} 
\end{msc}

\section{Introduction}
\label{sec:introduction}
We are interested in the numerical solution of
optimal control problems
of the type
\[
 \min_{u \in L^2(\Omega)} \frac12 \|S(u) - y_d\|_{L^2(\Omega)}^2
\]
subject to
\[
|u|\le 1 \text{ almost everywhere on }\Omega.
\]
Here, $\Omega$ is a bounded domain in $\R^d$ and $S:L^2(\Omega) \to L^2(\Omega)$ is the (possibly nonlinear) control-to-state mapping.
In addition, $y_d\in L^2(\Omega)$ is a given desired state.
For such problems,
optimal controls are often bang-bang,
i.e., at almost all points of the domain, optimal controls
only attain the control bounds $\pm 1$.
Typically, the first-order necessary and sufficient optimality condition
can be written in this case as
\begin{equation*}
	\bar u = \sign(-\bar p),
	\qquad
	\bar p = S'(\bar u)^*(S(\bar u)-y_d),
\end{equation*}
where $\bar u$ is an optimal control and $\bar p$ is the associated adjoint state.
In order to solve this system with Newton's method,
we need to show that the signum function $\sign$
is differentiable between appropriate function spaces.
Our main result \cref{thm:differentiability_signum}
shows that
$\sign \colon C^1(\bar\Omega) \to W^{-1,q}(\Omega)\dualspace$
is indeed differentiable at points $w \in C^1(\bar\Omega)$
satisfying a certain regularity condition.
We also show that this gives rise to local superlinear convergence of Newton's method
applied to an optimal control problem governed by linear or semilinear elliptic PDEs,
see \cref{sec:optimal_control}.
Numerical examples show the efficiency of our proposed methods,
see \cref{sec:numerics}.

Second-order optimality conditions for the control of semilinear equations
without a control regularization term
have been given in
\cite{Casas2012:1,ChristofWachsmuth2017:1,WachsmuthWachsmuth2022}.
A-priori error estimates for a variational discretization where given in \cite{DeckelnickHinze2012},
while \cite{Fuica2024,FuicaJork2025}
derived a-posteriori error estimates.
Finally,
we mention that second-order derivatives of the functional $\norm{\cdot}_{L^1(\Omega)}$ on $H_0^1(\Omega)$
have been analyzed in
\cite{Christof2018,ChristofMeyer2019}, which is related to our work,
since the signum function can be seen as the first derivative of $\norm{\cdot}_{L^1(\Omega)}$.

Let us report on existing approaches to solve bang-bang control problems.
Due to their simplicity, projected gradient \cite[Section 2.2.2]{HinzePinnauUlbrichUlbrich2009}  and conditional gradient methods \cite{KunischWalter2024} can be applied.
Frequently, fixed point methods are used to compute bang-bang controls \cite{DeckelnickHinze2012,Fuica2024,FuicaJork2025,VonDanielsHinze2020}.
With the help of the results of this paper, the convergence of fixed-point methods can be explained, see \cref{rem:fixed_point} below.
For controls
that only depend on time,
switching point optimization methods can be used \cite{GlashoffSachs1977,EpplerTroltzsch1986}.
Control problems for ordinary differential equations are amenable to the Goh transform and the transformed system can be solved by Newton's method \cite{Felgenhauer2016}.
Another popular approach is to apply Tikhonov regularization \cite{VonDaniels2018,WachsmuthWachsmuth2011}.

\section{Preliminaries}
\label{sec:prelim}
We use standard notation.
We denote by $\lambda$ the Lebesgue measure on $\R^d$
and by $\HH^{d-1}$ the $(d-1)$-dimensional Hausdorff measure.

For an open set $\Omega \subset \R^d$,
the vector space of all continuous (continuously differentiable) functions
from $\Omega$ to $\R$ is denoted by $C(\Omega)$ ($C^1(\Omega)$).
Lebesgue and Sobolev spaces are denoted by $L^p(\Omega)$ and $W^{k,p}(\Omega)$.
An index ``$c$'' at a function space indicates subspaces of functions with compact support.
We define $\sign \colon \R \to \R$ and $\Sign \colon \R \mto \R$ via
\begin{equation*}
	\sign(x) :=
	\begin{cases}
		-1 & \text{if } x < 0, \\
		0 & \text{if } x = 0, \\
		1 & \text{if } x > 0,
	\end{cases}
	\qquad
	\Sign(x) :=
	\begin{cases}
		\set{-1} & \text{if } x < 0, \\
		[-1,1] & \text{if } x = 0, \\
		\set{1} & \text{if } x > 0.
	\end{cases}
\end{equation*}
We mention that
open and closed balls
with radius $r > 0$
and center $x$
(in a metric space)
are denoted by
$U_r(x)$ and $B_r(x)$, respectively.

\subsection{Weak differentiability of the signum function}
\label{subsec:diff_signum}
We want to show that
$\sign$ is
directionally differentiable (in a weak sense)
at certain points in
$C^1(\Omega)$
w.r.t.\ directions having compact support.
First, we recall the differentiability result from \cite[Theorem~3.4]{WachsmuthWachsmuth2022}.
\begin{theorem}
	\label{thm:taylor_expansion_integral_rn}
	Let $U \subset \R^d$ be a bounded open set.
	Let $w \in C^1(U)$ be given such that
	$\nabla w \ne 0$ on $\set{w = 0}$.
	Further, let $z \in C_c(U)$ be given.
	For all $t > 0$, we define $U_t := \set{ \sign(w) \ne \sign(w + t z)}$.
	Then, for all $\psi \in C(U)$ we have
	\begin{subequations}
		\label{eq:diff_int_rn}
		\begin{align}
			\label{eq:diff_int_rn_a}
			\frac1t \int_{U_t} \psi \d\lambda &\to \int_{\set{w=0}}\frac{\psi \abs{z}}{\abs{\nabla w}} \d\HH^{d-1}
			,
			\\
			\label{eq:diff_int_rn_b}
			\frac1t \int_{U_t} \psi \sign(z) \d\lambda &\to \int_{\set{w=0}}\frac{\psi z}{\abs{\nabla w}} \d\HH^{d-1}
		\end{align}
	\end{subequations}
	as $t \searrow 0$.
\end{theorem}
This result can be used
to obtain a weak form of differentiability of the signum function
w.r.t.\ directions $z$ with compact support.
\begin{corollary}
	\label{cor:diff_signum}
	Let $U \subset \R^d$ be a bounded open set.
	Let $w \in C^1(U)$ be given such that
	$\nabla w \ne 0$ on $\set{w = 0}$.
	Further, let $z \in C^1_c(U)$ and $\psi \in C(U)$ be given.
	Then,
	\begin{subequations}
		\label{eq:diff_signum}
		\begin{align}
			\int_U
			\frac{
				\sign(w + t z)
				-
				\sign(w)
			}{t}
			\psi
			\d\lambda
			&\to
			2 \int_{\set{w = 0}} \frac{\psi z}{\abs{\nabla w}} \d\HH^{d-1}
			,
			\\
			\int_U
			\frac{
				\abs{
					\sign(w + t z)
					-
					\sign(w)
				}
			}{t}
			\psi
			\d\lambda
			&\to
			2 \int_{\set{w = 0}} \frac{\psi \abs{z}}{\abs{\nabla w}} \d\HH^{d-1}
		\end{align}
	\end{subequations}
	as $t \searrow 0$.
\end{corollary}
\begin{proof}
	Combining $\nabla w \ne 0$ on $\set{w = 0}$ with
	Stampacchia's lemma, we get $\nabla w = 0$ a.e.\ on $\set{w = 0}$.
	Thus, $\lambda(\set{ w = 0 }) = 0$.
	Further, the compactness of the support of $z$
	ensures that $\nabla (w + t z) \ne 0$ on $\set{w + t z = 0}$
	whenever $t$ is small enough.
	Consequently, $\lambda(\set{ w + t z = 0 }) = 0$.
	This implies that
	\begin{equation*}
		\sign(w + t z) - \sign(w) = 2 \chi_{U_t} \sign(z)
		\qquad\text{$\lambda$-a.e.\ on $\Omega$},
	\end{equation*}
	where $U_t$ is defined as in \cref{thm:taylor_expansion_integral_rn}.
	Thus, the claim follows from \eqref{eq:diff_int_rn}.
\end{proof}

\subsection{Extension of \texorpdfstring{$C^1$}{C1} functions}
\label{subsec:extension}

In this section,
we recall Whitney's extension theorem.
We start with some definitions,
see also \cite[Section~2]{Fefferman2007}.

\begin{definition}
	\label{def:function_spaces}
	Let $d \in \N$ be given.
	We define
	the function space
	\begin{equation*}
		\Cb^1(\R^d)
		:=
		\set*{
			f \in C^1(\R^d)
			\given
			\text{$f$ and $\nabla f$ are bounded}
		}
	\end{equation*}
	equipped with the usual supremum norm
	\begin{equation*}
		\norm{f}_{\Cb^1(\R^d)}
		:=
		\sup\set*{ \abs{\partial^\alpha f (x)} \given \abs{\alpha} \le 1, x \in \R^d }
		<
		\infty
		\qquad
		\forall f \in \Cb^1(\R^d)
		.
	\end{equation*}
	Let $\Omega \subset \R^d$ be open and bounded.
	We define
	\begin{equation*}
		C^1(\bar\Omega)
		:=
		\set{
			f \in C^1(\Omega)
			\given
			\partial^\alpha f
			\text{ uniformly continuous } \forall \abs{\alpha} \le 1
		}
	\end{equation*}
	with norm
	\begin{equation*}
		\norm{ f }_{C^1(\bar\Omega)}
		:=
		\sup\set*{ \abs{\partial^\alpha f (x)} \given \abs{\alpha} \le 1, x \in \Omega }
		<
		\infty
		\qquad\forall f \in C^1(\bar\Omega)
		.
	\end{equation*}
	For $f \in C^1(\bar\Omega)$
	we denote the uniquely determined continuous extensions of
	$f \colon \Omega \to \R$
	and
	$\nabla f \colon \Omega \to \R^d$
	to $\bar\Omega$
	also by
	$f \colon \bar\Omega \to \R$
	and
	$\nabla f \colon \bar\Omega \to \R^d$,
	respectively.

	Finally,
	for a compact set $K \subset \R^d$,
	we define $\Cjet^1(K)$
	as the set of all families
	$\seq{P^x}_{x \in K}$
	of affine functions $P^x \colon \R^d \to \R$
	satisfying
	\begin{enumerate}
		\item
			\label{def:jet:1}
			For all $\varepsilon > 0$ there exists $\delta > 0$
			such that
			\begin{equation*}
				\abs{ P^x(y) - P^y(y) } \le \varepsilon \abs{x - y}
				\qquad\text{and}\qquad
				\abs{ \partial_i P^x(y) - \partial_i P^y(y) } \le \varepsilon
			\end{equation*}
			holds for all $i \in \set{1,\ldots,d}$ and all  $x,y \in K$ with $\abs{x - y} < \delta$.
		\item
			\label{def:jet:2}
			There exists a constant $M \ge 0$ such that
			\begin{align*}
				\abs{P^x(x)} &\le M,
				&
				\abs{\partial_i P^x(x)} &\le M,
				\\
				\abs{P^x(y) - P^y(y)} &\le M \abs{x - y},
				\qquad\text{and}
				&
				\abs{\partial_i P^x(y) - \partial_i P^y(y)} &\le M
			\end{align*}
			holds for all $i \in \set{1,\ldots,d}$ and all $x,y \in K$.
	\end{enumerate}
	The norm $\norm{\seq{P^x}_{x \in K}}_{\Cjet^1(K)}$ is defined to be the smallest constant $M$
	satisfying these inequalities.
\end{definition}

\begin{definition}
	\label{def:LUQ}
	We say that a set $\Omega \subset \R^d$ is
	\emph{uniformly locally quasiconvex}
	with constants
	$r > 0$ and $L \ge 1$
	if for any
	$x,y \in \Omega$ with $\abs{y - x} \le r$
	there exists a Lipschitz continuous curve
	$\gamma \in C^{0,1}([0,1]; \Omega)$
	with $\gamma(0) = x$, $\gamma(1) = y$
	and Lipschitz constant
	at most $L \abs{y - x}$.
\end{definition}

We recall the following result from \cite[Lemma~5.3]{WachsmuthWalter2024}.
\begin{lemma}
	\label{lem:Lipschitz_implies_ulq}
	Let $\Omega \subset \R^d$ be open and bounded.
	Further, we require that $\Omega$
	is a Lipschitz set
	in the sense that
	for every point $p \in \partial \Omega$,
	there exists $r > 0$ and
	a bijection $l_p \colon B_r(p) \to B_1(0)$
	such that $l_p$ and $l_p^{-1}$ are Lipschitz,
	\begin{equation*}
		l_p( \Omega \cap B_r(p)) = \set{z \in B_1(0) \given z_n > 0},
		\quad\text{and}\quad
		l_p( \partial \Omega \cap B_r(p)) = \set{z \in B_1(0) \given z_n = 0}
		.
	\end{equation*}
	Then,
	$\Omega$ is uniformly locally quasiconvex.
\end{lemma}

As a further ingredient,
we need the celebrated extension theorem by Whitney
\cite[Theorem~I]{Whitney1934:2}
which we state as in \cite[Theorem~2.1]{Fefferman2007}.
\begin{theorem}
	\label{thm:whithney}
	Given a compact set $K \subset \R^d$,
	there exists a bounded linear map $E_K \colon \Cjet^1(K) \to \Cb^1(\R^d)$
	such that
	\begin{equation*}
		P^x(y)
		=
		E_K(\seq{P^x}_{x \in K})(x)
		+
		\nabla [E_K(\seq{P^x}_{x \in K})](x)^\top (y - x)
		\qquad
		\forall x \in K, y \in \R^d
	\end{equation*}
	for all $\seq{P^x}_{x \in K} \in \Cjet^1(K)$.
	Moreover, the norm of $E_K$ only depends on the dimension $d$.
\end{theorem}

By combining the previous results, we get our desired extension result.
\begin{theorem}
	\label{thm:C1_extension}
	Suppose that $\Omega \subset \R^d$
	is open, bounded, and uniformly locally quasiconvex.
	Then, there exists a bounded linear operator $E \colon C^1(\bar\Omega) \to \Cb^1(\R^d)$
	with
	\begin{equation*}
		(E f)_{|\bar\Omega} = f
	\end{equation*}
	for all $f \in C^1(\bar\Omega)$.
	Note that this also implies that the derivatives of $E f$ and $f$ coincide on $\bar\Omega$.
\end{theorem}
\begin{proof}
	We follow the proof of \cite[Theorem, p. 485]{Whitney1934:1}.

	We
	associate with $f$ the jet of its Taylor polynomials.
	That is, we set
	\begin{equation*}
		P^x(y) := f(x) + \nabla f(x)^\top (y - x).
	\end{equation*}
	We are going to check that
	$\seq{P^x}_{x \in \bar\Omega} \in \Cjet^1(\bar\Omega)$.
	First of all, it is clear that
	\begin{equation*}
		\abs{P^x(x)} = \abs{f(x)} \le \norm{f}_{C^1(\bar\Omega)},
		\qquad
		\abs{\partial_i P^x(x)} = \abs{\partial_i f(x)} \le \norm{f}_{C^1(\bar\Omega)}
		,
	\end{equation*}
	and
	\begin{equation*}
		\abs{\partial_i P^x(y) - \partial_i P^y(y)} \le 2 \norm{f}_{C^1(\bar\Omega)}
		.
	\end{equation*}

	We assume that $\Omega$ uniformly locally quasiconvex
	with constants $r > 0$ and $L \ge 1$.

	Next, let $x,y \in \Omega$
	with $\abs{x - y} < r$ be given.
	We denote the curve connecting $x$ and $y$ by $\gamma$.
	Then,
	\begin{align*}
		\abs{
			P^x(y) - P^y(y)
		}
		&=
		\abs{
			f(x) - f(y) + \nabla f(x)^\top (y - x)
		}
		\\
		&=
		\abs*{
			\int_0^1 [\nabla f(\gamma(t)) - \nabla f(x)]^\top \gamma'(t) \d t
		}
		\\
		&\le
		L \abs{x - y} \int_0^1 \abs{ \nabla f(\gamma(t)) - \nabla f(x) } \d t
		\\
		&\le
		2 L \abs{x - y} \norm{f}_{C^1(\bar\Omega)}
		.
	\end{align*}
	By continuity, the same estimate holds for
	$x, y \in \bar\Omega$ with $\abs{x - y} < r$.
	For $x, y \in \bar\Omega$ with $\abs{x - y} \ge r$
	we can simply use
	\begin{equation*}
		\abs{
			P^x(y) - P^y(y)
		}
		\le
		(2 + \diam(\Omega))
		\norm{f}_{C^1(\bar\Omega)}
		\le
		\frac{2+ \diam(\Omega)}{r} \norm{f}_{C^1(\bar\Omega)} \abs{x - y}
		.
	\end{equation*}
	This shows that
	property \ref{def:jet:2}
	of \cref{def:function_spaces} holds
	with $M = C \norm{f}_{C^1(\bar\Omega)}$,
	where the constant $C$ only depends on $\Omega$
	via $r$, $L$, and $\diam(\Omega)$.

	We denote by $\zeta$ the modulus of continuity of $\nabla f$ on $\bar\Omega$.
	Let again $x, y \in \Omega$ with $\abs{x - y} < r$
	be given and let $\gamma$ be a curve connecting $x$ and $y$.
	This enables us to estimate
	\begin{align*}
		\abs{
			P^x(y) - P^y(y)
		}
		&\le
		L \abs{x - y} \int_0^1 \abs{ \nabla f(\gamma(t)) - \nabla f(x) } \d t
		\\
		&\le
		L \abs{x - y} \zeta( L \abs{x - y} )
	\end{align*}
	and
	\begin{equation*}
		\abs{
			\partial_i P^x(y) - \partial_i P^y(y)
		}
		\le
		\zeta( L \abs{x - y} ).
	\end{equation*}
	By continuity, these estimates can be extended
	to $x,y \in \bar\Omega$.
	Consequently, property \ref{def:jet:1}
	of \cref{def:function_spaces} holds.
	Hence,
	$\seq{P^x}_{x \in \bar\Omega} \in \Cjet^1(\bar\Omega)$
	and
	\begin{equation*}
		\norm{ \seq{P^x}_{x \in \bar\Omega} }_{\Cjet^1(\bar\Omega)}
		\le
		C \norm{ f }_{C^1(\bar\Omega)}.
	\end{equation*}
	Finally, the jet can be extended via
	\cref{thm:whithney} and this yields the claim.
\end{proof}

\section{Continuous differentiability of the signum function}
\label{sec:cty_derivative_signum}
In this section,
we study
the function $\sign \colon C^1(\bar\Omega) \to W^{1,q}(\Omega)\dualspace$,
defined via
\begin{equation}
	\label{eq:signum}
	\dual{\sign( v )}{ \psi }_{W^{1,q}(\Omega)}
	:=
	\int_\Omega \sign(v) \psi \d\lambda
	\qquad
	\forall v \in C^1(\bar\Omega), \psi \in W^{1,q}(\Omega)
	.
\end{equation}
The exponent $q > 1$ is fixed throughout this section.
We are going to show that $\sign$ is
continuously
differentiable at points $w \in C^1(\bar\Omega)$
satisfying $\nabla w = 0$ on $\set{w = 0}$
and \eqref{eq:asm_bdry} below,
and the derivative is given by
\begin{equation}
	\label{eq:derivative_signum}
	\dual{\sign'( w ) z }{ \psi }_{W^{1,q}(\Omega)}
	:=
	2 \int_{\set{w = 0}} \frac{\psi z}{\abs{\nabla w}} \d\HH^{d-1}
	\qquad
	\forall z \in C^1(\bar\Omega), \psi \in W^{1,q}(\Omega).
\end{equation}

We use the following standing assumption on the domain $\Omega$.
\begin{assumption}[Standing assumption]
	\label{asm:standing}
	Suppose that $\Omega \subset \R^d$
	is open, bounded, and uniformly locally quasiconvex, see \cref{def:LUQ}.
	Further, we assume that $\Omega$ is an $(1,q)$-extension domain,
	i.e., that there exists a linear and bounded extension operator
	$E_W \colon W^{1,q}(\Omega) \to W^{1,q}(\R^d)$.
\end{assumption}
Throughout this section,
we denote by $E$ the extension operator given in \cref{thm:C1_extension}.
By $C_E$ we denote its operator norm.

We note that the sets of the form $\set{w = 0}$ depend on the domain of the function $w$.
In particular, for $w \in C^1(\bar\Omega)$ we have
$\set{w = 0} \subset \bar\Omega$
and
$\set{E w = 0} \subset \R^d$.

We specify the meaning of the integral in \eqref{eq:derivative_signum}
by showing that the trace of a Sobolev function on $\set{w = 0}$ is well defined.

\begin{lemma}
	\label{lem:trace}
	Let $w \in C^1(\bar\Omega)$ with $\nabla w \ne 0$ on $\set{w = 0}$ be given.
	Then,
	\begin{enumerate}
	 \item\label{it_trace_1} there is a bounded and linear operator $T_{\R^d} \colon W^{1,q}(\R^d) \to L^{q}(\HH^{d-1}|_{\set{w = 0}})$
	such that $T_{\R^d} \psi = \psi|_{\set{w = 0}}$
	for all $\psi \in C(\R^d) \cap W^{1,q}(\R^d)$,

	\item\label{it_trace_2}  there is a bounded and linear operator $T_\Omega \colon W^{1,q}(\Omega) \to L^{q}(\HH^{d-1}|_{\set{w = 0}})$
	such that
	$T_\Omega \psi = \psi|_{\set{w = 0}}$
	for all $\psi \in C(\bar\Omega) \cap W^{1,q}(\Omega)$,

	\item\label{it_trace_3} and $T_\Omega = T_{\R^d} E_W$.
	\end{enumerate}
\end{lemma}
\begin{proof}
\ref{it_trace_1}
	The set $\set{w = 0} \subset \bar\Omega$ is a compact subset
	of the $C^1$-hypersurface $\set{ E w = 0} \cap \set{ \nabla E w \ne 0}$.
	Consequently, one can use the usual arguments to check
	the existence of $C_T \ge 0$
	such that
	\begin{equation*}
		\int_{\set{w = 0}} \abs{\psi}^q \d\HH^{d-1}
		\le
		C_T^q \norm{\psi}_{W^{1,q}(\R^d)}^q
		\qquad
		\forall
		\psi \in C^1(\R^d) \cap W^{1,q}(\R^d).
	\end{equation*}
	This gives rise to a continuous trace operator
	from
	$C^1(\R^d) \cap W^{1,q}(\R^d)$ (equipped with the norm of $W^{1,q}(\R^d)$)
	into
	$L^q(\HH^{d-1}|_{\set{w = 0}})$.
	By density and continuity, this operator can be extended to all of $\psi \in W^{1,q}(\R^d)$, which gives rise to the trace operator $T_{\R^d}$.
	For $\psi \in C(\R^d) \cap W^{1,q}(\R^d)$, the identity $T_{\R^d} \psi = \psi|_{\set{w = 0}}$ follows by a mollification argument.
	This shows \ref{it_trace_1}.

We define  $T_\Omega := T_{\R^d} E_W$, so that \ref{it_trace_3} and the boundedness claim of  \ref{it_trace_2} follow trivially.
Let us show that  $T_\Omega\psi$ coincides with $\psi|_{\set{w = 0}}$ for all $\psi \in C(\bar\Omega) \cap W^{1,q}(\Omega)$.

First, we will argue as in \cite[Remark~4.4.5]{Ziemer1989} to prove that $E_W\psi(x_0) = \psi(x_0)$ for  $\HH^{d-1}$-a.a. $x_0\in\set{w=0}$.
As shown there, $\HH^{d-1}$-almost every $x_0 \in \{w=0\}$ is a Lebesgue point of $E_W \psi$.
Take such a Lebesgue point $x_0 \in \{w=0\}$. Then \cite[Remark~4.4.5, page 190]{Ziemer1989}
constructs a measurable set $A\subset \R^d$ such that $\lim_{A \ni x \to x_0} (E_W\psi)(x) = (E_W\psi)(x_0)$
and $\lim_{r\searrow0} \frac{\lambda(B(x_0,r)\cap A)} {\lambda(B(x_0,r))}=1$.
Since $\Omega$ is assumed to be an $(1,q)$-extension domain, \cite[Theorem~2]{HajlaszKoskelaTuominen2008} implies that $\Omega$
possesses a uniformly positive Lebesgue density at all points, which implies that $\Omega$
possesses a positive Lebesgue density at all points
of the boundary $\partial\Omega$, i.e., $\lim_{r\searrow0} \frac{\lambda(B(x_0,r)\cap \Omega)} {\lambda(B(x_0,r))}=:c>0$.
Due to $\chi_A \chi_\Omega \ge \chi_A+ \chi_\Omega -1$, it follows $\lim_{r\searrow0} \frac{\lambda(B(x_0,r)\cap \Omega \cap A)} {\lambda(B(x_0,r))}=c>0$.
In particular, $B(x_0,r)\cap \Omega \cap A$ is non-empty for all $r>0$.
We conclude
\begin{equation}\label{eq_magic_trace_identity}
 (E_W\psi)(x_0) = \lim_{\substack{x\to x_0\\x\in A}} (E_W\psi)(x) = \lim_{\substack{x\to x_0\\x\in A\cap \Omega}} (E_W\psi)(x)
 = \psi(x_0),
\end{equation}
since $E_W \psi$ coincides with $\psi$ a.e.\ on $\Omega$.

Second, we prove $(T_{\R^d} E_W\psi)(x_0)=(E_W\psi)(x_0) $ for  $\HH^{d-1}$-a.a. $x_0\in \set{w=0}$.
To this end, we use a mollification argument. Let $\rho_\epsilon$ denote the standard mollification kernel.
Then $\rho_\epsilon \mathbin{*} (E_W\psi) \to E_W\psi$ in $W^{1,q}(\R^d)$ for $\epsilon \to0$ and
$(\rho_\epsilon \mathbin{*} (E_W\psi))(x) \to E_W\psi(x)$ for all Lebesgue points of $E_W\psi$, \cite[Section 4.2, Theorem 1(iv)]{EvansGariepy1992}.
From claim \ref{it_trace_1}, we get $T_{\R^d}(\rho_\epsilon \mathbin{*} (E_W\psi)) \to T_{\R^d}E_W\psi = T_\Omega\psi$ in $L^q(\HH^{d-1}|_{\set{w = 0}})$.
In addition, $T_{\R^d}(\rho_\epsilon \mathbin{*} (E_W\psi)) = \rho_\epsilon \mathbin{*} (E_W\psi)|_{\set{w=0}}$.
Passing to the limit $\epsilon\searrow0$ in this identity implies $T_\Omega\psi = E_W\psi$ $\HH^{d-1}$-a.e. on $\set{w=0}$. Together with \eqref{eq_magic_trace_identity} this proves
$T_\Omega\psi(x_0) = \psi(x_0)$ for $\HH^{d-1}$-a.a. $x_0 \in \set{w=0}$,
which finishes the proof of \ref{it_trace_2}.
\end{proof}

In the sequel, we will follow the common practice to suppress the appearance of the trace operators, which is justified by the previous result.

For $\eta > 0$,
we define a neighborhood of the boundary via
\begin{equation*}
	\partial_\eta\Omega := \set{ x \in \R^d \given \dist(x, \partial\Omega) < \eta }
\end{equation*}
and
a neighborhood of $\bar\Omega$ via
\begin{equation*}
	\Omega_\eta
	:=
	\partial_\eta \Omega \cup \Omega
	=
	\set{x \in \R^d \given \dist(x, \Omega) < \eta}
	.
\end{equation*}

We start with an estimate in the neighborhood of the boundary.
\begin{lemma}
	\label{lem:cty_bdry}
	Let $w \in C^1(\bar\Omega)$ with $\nabla w \ne 0$ on $\set{w = 0}$ be given
	such that
	\begin{equation}
		\label{eq:asm_bdry}
		\HH^{d-1}\parens*{
			\set{w = 0} \cap \partial\Omega
		}
		=
		0
		.
	\end{equation}
	Then, for every $\varepsilon > 0$,
	there exists $\eta > 0$ such that
	$\nabla (E w) \ne 0$ on $\set{E w = 0} \cap \overline{\partial_\eta \Omega}$
	and
	\begin{equation*}
		\int_{\set{E w = 0} \cap \partial_\eta \Omega } \frac{\psi}{\abs{\nabla (E w)}} \d\HH^{d-1}
		\le
		\varepsilon
		\norm{\psi}_{W^{1,q}(\R^d)}
		\qquad\forall \psi \in W^{1,q}(\R^d)
		.
	\end{equation*}
	Here, the trace of $\psi$ on $\set{E w = 0} \cap \partial_\eta \Omega$
	has to be understood similarly to \cref{lem:trace}.
\end{lemma}
\begin{proof}
	By compactness of $\bar\Omega$,
	there exists a constant $\nu > 0$ such that
	$\abs{\nabla w} \ge \nu$ on $\set{w = 0}$.
	Consequently,
	$\abs{\nabla (E w)} \ge \nu$ on $\set{E w = 0} \cap \bar\Omega$.
	Using continuity of $E w$ and $\nabla (E w)$,
	this shows the existence of $\eta_0 > 0$ with
	$\abs{\nabla (E w)} \ge \nu / 2$ on $\set{E w = 0} \cap \overline{\partial_{\eta_0} \Omega}$.
	Then in a neighborhood of $\overline{\partial_{\eta_0} \Omega}$, $\set{E w = 0}$ is locally the graph of a continuously differentiable function.
	Since the set $\overline{\partial_{\eta_0} \Omega}$ is compact,
	it follows $\HH^{d-1}(\set{E w = 0} \cap \overline{\partial_{\eta_0} \Omega})<\infty$.
	Consequently,
	the dominated convergence theorem implies
	\begin{equation*}
		\HH^{d-1}\parens{\set{E w = 0} \cap \partial_\eta \Omega }
		\to
		\HH^{d-1}\parens{\set{E w = 0} \cap \partial \Omega}
		=
		\HH^{d-1}\parens{\set{w = 0} \cap \partial \Omega}
		=
		0
	\end{equation*}
	as $\eta \to 0$.
	Using the analogue of \cref{lem:trace},
	we get the trace estimate
	\begin{equation*}
		\norm{\psi}_{L^p(\HH^{d-1}|_{\set{E w = 0} \cap \Omega_{\eta_0}})}
		=
		\parens*{
			\int_{\set{E w = 0} \cap \Omega_{\eta_0}} \abs{\psi}^q \d\HH^{d-1}
		}^{1/q}
		\le
		C \norm{\psi}_{W^{1,q}(\R^d)}
	\end{equation*}
	for all $\psi \in W^{1,q}(\R^d)$.
	Consequently,
	\begin{align*}
		\int_{\set{E w = 0} \cap \partial_{\eta} \Omega} \frac{\psi}{\abs{\nabla (E w)}} \d\HH^{d-1}
		&\le
		\frac{2}{\nu} \HH^{d-1}(\set{E w = 0} \cap \partial_\eta \Omega)^{1/q'} C \norm{\psi}_{W^{1,q}(\R^d)}
		\\&
		\le
		\varepsilon \norm{\psi}_{W^{1,q}(\R^d)}
		\qquad
		\forall
		\psi \in W^{1,q}(\R^d)
	\end{align*}
	for $\eta \in (0,\eta_0)$ small enough.
\end{proof}
By combining \cref{lem:cty_bdry}
with \cref{cor:diff_signum}
we get a weak form of differentiability.
\begin{lemma}
	\label{lem:differentiability_signum}
	Let $w \in C^1(\bar\Omega)$ with $\nabla w \ne 0$ on $\set{w = 0}$ be given
	such that \eqref{eq:asm_bdry} holds.
	Then, for every $\psi \in C(\bar\Omega) \cap W^{1,q}(\Omega)$,
	the function
	$C^1(\bar\Omega) \ni v \mapsto \dual{\sign(v)}{\psi}_{W^{1,q}(\Omega)} \in \R$
	is
	Gâteaux differentiable at $w$
	with derivative
	$C^1(\bar\Omega) \ni z \mapsto \dual{\sign'( w ) z }{ \psi }_{W^{1,q}(\Omega)} \in \R$.
\end{lemma}
\begin{proof}
	Let $\varepsilon > 0$ be given
	and choose $\eta > 0$ according to \cref{lem:cty_bdry}.
	Next, we choose functions
	$\varphi_1 \in C_c^\infty(\Omega; [0,1])$
	and
	$\varphi_2 \in C_c^\infty(\partial_\eta \Omega; [0,1])$
	such that
	\begin{equation*}
		\varphi_1 = 1 \text{ on } \Omega \setminus \partial_{\eta/2}\Omega
		\quad\text{and}\quad
		\varphi_2 = 1 \text{ on } \Omega \cap \set{ \varphi_1 < 1 }
		.
	\end{equation*}
	This construction implies $1-\varphi_1 \le \varphi_2$ on $\Omega \cap \partial_\eta \Omega$.

	We further fix $z \in C^1(\bar\Omega)$ and $\psi \in C(\bar\Omega) \cap W^{1,q}(\Omega)$.
	We have
	\begin{align*}
		I &:=
		\frac{
			\dual{\sign(w + t z)}{\psi}_{W^{1,q}(\Omega)}
			-
			\dual{\sign(w + t z)}{\psi}_{W^{1,q}(\Omega)}
		}{t}
		-
		\dual{\sign'(w) z}{\psi}_{W^{1,q}(\Omega)}
		\\
		&=
		\int_\Omega \frac{\sign(w + t z) - \sign(w)}{t} \psi \d\lambda
		-
		2 \int_{\set{w = 0}} \frac{\psi z}{\abs{\nabla w}} \d\HH^{d-1}
		\\
		&=
		\parens*{
			\int_\Omega \frac{\sign(w + t z) - \sign(w)}{t} \psi \varphi_1 \d\lambda
			-
			2 \int_{\set{w = 0}} \frac{\psi \varphi_1 z}{\abs{\nabla w}} \d\HH^{d-1}
		}
		\\
		&\qquad+
		\parens*{
			\int_\Omega \frac{\sign(w + t z) - \sign(w)}{t} \psi (1 - \varphi_1) \d\lambda
			-
			2 \int_{\set{w = 0}} \frac{\psi (1 - \varphi_1) z}{\abs{\nabla w}} \d\HH^{d-1}
		}
		\\
		&=:
		I_1 + I_2
		.
	\end{align*}
	Since $\varphi_1$ has compact support in $\Omega$, there is $\varphi_3 \in C_c^\infty( \Omega; [0,1])$ with $\varphi_3 =1$ on $\supp \varphi_1$.
	From \cref{cor:diff_signum} applied to $z := z\varphi_3 \in C^1_c(\Omega)$
	and $\psi := \psi \varphi_1 \in C(\Omega)$,
	we get $\abs{I_1} \to 0$ as $t \searrow 0$.
	For the second term, we use
	\begin{align*}
		\abs{I_2}
		&\le
		\int_\Omega \frac{\abs{\sign(w + t z) - \sign(w)}}{t} \abs{\psi} (1 - \varphi_1) \d\lambda
		+
		2 \int_{\set{w = 0}} \frac{\abs{\psi} (1 - \varphi_1) \abs{z}}{\abs{\nabla w}} \d\HH^{d-1}
		\\
		&\le
		\parens*{
			\int_{\partial_\eta \Omega} \frac{\abs{\sign(E w + t E z) - \sign(E w)}}{t} \varphi_2 \d\lambda
			+
			2 \int_{\set{E w = 0}} \frac{ \varphi_2 \abs{E z}}{\abs{\nabla (E w)}} \d\HH^{d-1}
		}
		\norm{\psi}_{C(\bar\Omega)}
		\\
		&=: (I_3 + I_4) \norm{\psi}_{C(\bar\Omega)}
		.
	\end{align*}
	The expression $I_4$ can be bounded via \cref{lem:cty_bdry}, i.e.,
	\begin{equation*}
		I_4
		\le
		2 \norm{E z}_{\Cb^1(\R^d)}
		\int_{\set{E w = 0} \cap \partial_\eta \Omega } \frac{ \varphi_2 }{\abs{\nabla (E w)}} \d\HH^{d-1}
		\le
		2 C_E \norm{z}_{C^1(\bar\Omega)} \norm{\varphi_2}_{W^{1,q}(\R^d)} \varepsilon.
	\end{equation*}
	The term involving $I_3$
	can be handled by applying \cref{cor:diff_signum}
	on $U = \partial_\eta \Omega$
	and we obtain
	\begin{equation*}
		I_3
		\to
		2 \int_{\set{E w = 0} \cap \partial_\eta \Omega} \frac{\varphi_2 \abs{E z}}{\abs{\nabla (E w)}} \d\HH^{d-1}
		=
		I_4
		\le
		2 C_E \norm{z}_{C^1(\bar\Omega)} \norm{\varphi_2}_{W^{1,q}(\R^d)} \varepsilon
	\end{equation*}
	as $t \searrow 0$.
	Combining the above estimates shows that $\abs{I}$
	can be made arbitrarily small by choosing $t > 0$ small enough.
	This shows that $v \mapsto \dual{\sign(v)}{\psi}_{W^{1,q}(\Omega)}$
	is directionally differentiable at $w$.
	Since the directional derivative is linear w.r.t.\ the direction $z$,
	this yields the desired Gâteaux differentiability.
\end{proof}
Note that the proof only uses the regularity $\psi \in C(\bar\Omega)$
for the estimates.
We only required $\psi \in W^{1,q}(\Omega)$,
since $\sign'(w) z$ was defined to be an element in the dual space of $W^{1,q}(\Omega)$,
see \eqref{eq:derivative_signum}.

As a next goal, we want to prove the continuity of the derivative
$\sign'$.
The next lemma guarantees that \cref{lem:trace}
can also be applied to $\tilde w$ in the neighborhood of $w$.
\begin{lemma}
	\label{lem:stability_of_reg_condition}
	Let $w \in C^1(\bar\Omega)$ with $\nabla w \ne 0$ on $\set{w = 0}$ be given.
	Then there is $\delta>0$ such that for all $\tilde w \in C^1(\bar\Omega)$ with $\norm{w - \tilde w}_{C^1(\bar\Omega)} \le \delta $
	it follows $\nabla\tilde w \ne 0$ on $\set{\tilde w = 0}$.
\end{lemma}
\begin{proof}
 Since $\set{w=0}$ is compact, there is $\delta_1$ such that $\abs{\nabla w} \ge \delta_1$ on $\set{w=0}$.
 By continuity of $\nabla w$, there is $U \subset \R^d$ open such that $U \supset \set{w=0}$ and $\abs{\nabla w} \ge \delta_1 / 2$ on $U$.
 Then there is $\delta_2>0$ such that $\abs{ w} \ge \delta_2$ on the compact set $\bar \Omega \setminus U$.
 The claim now follows with $0<\delta< \min(\delta_1/2,\delta_2)$.
\end{proof}

The next results gives almost the desired continuity,
but involves the extended functions and a cutoff function.
\begin{lemma}
	\label{lem:cty_for_extensions}
	Let $w \in C^1(\bar\Omega)$ with $\nabla w \ne 0$ on $\set{w = 0}$
	be given.
	Further, assume that $\eta_0 > 0$ is chosen such that
	$\nabla (E w) \ne 0$ on $\set{E w = 0} \cap \overline{\Omega_{\eta_0}}$.
	Then for
	all
	$\varepsilon > 0$
	and $\eta \in (0,\eta_0)$,
	there exists $\delta > 0$ such that
	\begin{equation*}
		\abs*{
			\int_{\set{E w = 0}} \frac{\psi}{\abs{\nabla (E w)}} \d \HH^{d-1}
			-
			\int_{\set{E \tilde w = 0}} \frac{\psi}{\abs{\nabla (E \tilde w)}} \d \HH^{d-1}
		}
		\le
		\varepsilon
		\norm{\psi}_{W^{1,q}(\Omega_{\eta_0})}
	\end{equation*}
	for all $\tilde w \in C^1(\bar\Omega)$ with $\norm{w - \tilde w}_{C^1(\bar\Omega)} \le \delta $
	and
	all $\psi \in W_0^{1,q}(\Omega_{\eta_0})$.
\end{lemma}
\begin{proof}
	In this proof, it will be convenient
	to supress the extension operator $E$.
	Note that \cref{thm:C1_extension}
	shows that
	$\norm{E w - E \tilde w}_{\Cb^1(\R^d)}$
	can be bounded
	by a multiple of
	$\norm{w - \tilde w}_{C^1(\bar\Omega)}$.
	That is, we assume that
	$w \in \Cb^1(\R^d)$
	satisfies
	$\nabla w \ne 0$ on $\set{w = 0} \cap \overline{\Omega_{\eta_0}}$
	and prove
	\begin{equation*}
		\abs*{
			\int_{\set{w = 0}} \frac{\psi}{\abs{\nabla w}} \d \HH^{d-1}
			-
			\int_{\set{\tilde w = 0}} \frac{\psi}{\abs{\nabla \tilde w}} \d \HH^{d-1}
		}
		\le
		\varepsilon
		\norm{\psi}_{W^{1,q}(\Omega_{\eta_0})}
	\end{equation*}
	for all $\tilde w \in \Cb^1(\R^d)$ with $\norm{w - \tilde w}_{\Cb^1(\R^d)} \le \delta $
	and
	all $\psi \in W_0^{1,q}(\Omega_{\eta_0})$.

	Since $\overline{\Omega_{\eta_0}}$ is compact,
	$\nabla w$ is uniformly continuous on this set and
	we denote by $\zeta$ its modulus of continuity.
	Let $\eta \in (0, \eta_0)$
	be arbitrary.

	We begin with a local argument.
	Let a point $p \in \set{w = 0} \cap \Omega_{\eta_0}$ with $\partial_d w(p) > 0$
	be given.
	Due to the implicit function theorem,
	there exist
	an open set $B \subset \R^{d-1}$,
	an open interval $I = (a,b) \subset \R$,
	a function $\gamma \in C^1(B; I)$
	and a constant $\nu \in (0,1]$
	such that
	\begin{align*}
		& p \in B \times I
		, \qquad
		\overline{B \times I} \subset \Omega_{\eta_0}
		, \qquad
		\partial_d w(x,y) \ge \nu \quad\forall (x,y) \in B \times I
		\\
		& \set{w = 0} \cap (B \times I) = \set{(x, \gamma(x)) \given x \in B},
		\\ &
		\set{(x,y) \given x \in B, \abs{y - \gamma(x)} < \nu } \subset B \times I
		.
	\end{align*}
	Note that this implies
	$w(x,a) \le -\nu^2$ and $w(x,b) \ge \nu^2$
	for all $x \in B$.
	We define $\delta := \nu^2/2$.
	For an arbitrary $\tilde w \in \Cb^1(\R^d)$
	with
	$\norm{w - \tilde w}_{\Cb^1(\R^d)} \le \delta$,
	we therefore obtain
	$\tilde w(x,a) \le -\delta$ and $w(x,b) \ge \delta$
	for all $x \in B$
	and
	$\partial_d \tilde w(x,y) > \nu/2$ for all $(x,y) \in B \times I$.
	Consequently,
	the intermediate value theorem in cooperation with
	the implicit function theorem implies the existence of
	$\tilde\gamma \in C^1(B; I)$
	such that
	\begin{equation*}
		\set{\tilde w = 0} \cap (B \times I)
		=
		\set{(x, \tilde\gamma(x)) \given x \in B}
		.
	\end{equation*}
	From
	$ \abs{ w(x, \tilde\gamma(x)) } = \abs{ w(x, \tilde\gamma(x)) - \tilde w(x, \tilde\gamma(x)) }$,
	$w(x,\gamma(x)) = 0$,
	and
	$\partial_d w(x,y) \ge \nu$,
	we get
	\begin{equation*}
		\abs{
			\gamma(x) - \tilde\gamma(x)
		}
		\le
		\frac{\norm{w - \tilde w}_{\Cb(\R^d)}}{\nu}
		.
	\end{equation*}
	Now, let
	$\psi \in C_c^1( B \times I )$
	be arbitrary.
	By differentiating $w(x,\gamma(x)) = 0$,
	we obtain
	$-\partial_i w(x,\gamma(x)) = \partial_d w(x,\gamma(x)) \partial_i \gamma(x)$
	for all $i = 1,\ldots, d-1$.
	Thus,
	\begin{equation*}
		\abs{\nabla w(x,\gamma(x))}^2
		=
		\partial_d w(x,\gamma(x))^2 (\abs{\gamma'(x)}^2 + 1).
	\end{equation*}
	From the area formula, we get
	\begin{align*}
		J &:=
		\int_{\set{w = 0}} \frac{\psi}{\abs{\nabla w}} \d\HH^{d-1}
		=
		\int_{\set{w = 0} \cap (B \times I)} \frac{\psi}{\abs{\nabla w}} \d\HH^{d-1}
		\\&=
		\int_B \frac{\psi(x, \gamma(x))}{ \abs{\nabla w(x,\gamma(x)) }} \sqrt{\abs{\gamma'(x)}^2 + 1} \d\lambda^{d-1}(x)
		=
		\int_B \frac{\psi(x, \gamma(x))}{ \partial_d w(x,\gamma(x)) } \d\lambda^{d-1}(x)
		,
	\end{align*}
	where $\lambda^{d-1}$ is the Lebesgue measure on $\R^{d-1}$.
	The same argument shows
	\begin{equation*}
		\tilde J
		:=
		\int_{\set{\tilde w = 0}} \frac{\psi}{\abs{\nabla \tilde w}} \d\HH^{d-1}
		=
		\int_B \frac{\psi(x, \tilde\gamma(x))}{ \partial_d \tilde w(x,\tilde\gamma(x)) } \d\lambda^{d-1}(x)
		.
	\end{equation*}
	Taking the difference gives
	\begin{align*}
		\abs{J - \tilde J}
		&\le
		\int_B \frac{\abs{\psi(x, \gamma(x)) - \psi(x, \tilde\gamma(x))}}{ \partial_d \tilde w(x,\tilde\gamma(x)) } \d\lambda^{d-1}(x)
		\\&\qquad
		+
		\int_B \psi(x, \gamma(x)) \abs*{
			\frac{1}{ \partial_d w(x,\gamma(x)) }
			-
			\frac{1}{ \partial_d \tilde w(x,\tilde\gamma(x)) }
		} \d\lambda^{d-1}(x)
		\\
		&:=
		I_1 + I_2.
	\end{align*}
	For the first contribution, we get
	\begin{align*}
		I_1
		&\le
		\frac2\nu
		\int_B \int_{\gamma(x)}^{\tilde\gamma(x)}\abs{\partial_d \psi(x, t)} \d t \d\lambda^{d-1}(x)
		\\
		&\le
		\frac2\nu
		\int_B \abs{\gamma(x) - \tilde\gamma(x)}^{1/q'} \parens*{\int_{a}^{b}\abs{\partial_d \psi(x, t)}^q \d t}^{1/q} \d\lambda^{d-1}(x)
		\\
		&\le
		\frac2\nu
		\parens*{ \norm{w - \tilde w}_{\Cb(\R^d)} / \nu }^{1/q'}
		\int_B \parens*{\int_{a}^{b}\abs{\partial_d \psi(x, t)}^q \d t}^{1/q} \d\lambda^{d-1}(x)
		\\
		&\le
		\frac{2}{\nu}
		\parens*{ \norm{w - \tilde w}_{\Cb(\R^d)} / \nu }^{1/q'}
		\lambda^{d-1}(B)^{1/q'}
		\norm{ \psi }_{W^{1,q}(B \times I)}
		.
	\end{align*}
	For the second contribution, we start with the estimate
	\begin{align*}
		\abs{
			\partial_d w(x,\gamma(x))
			-
			\partial_d \tilde w(x,\tilde\gamma(x))
		}
		&\le
		\abs{
			\partial_d w(x,\gamma(x))
			-
			\partial_d w(x,\tilde\gamma(x))
		}
		\\&\qquad
		+
		\abs{
			\partial_d w(x,\tilde\gamma(x))
			-
			\partial_d \tilde w(x,\tilde\gamma(x))
		}
		\\
		&\le
		\zeta(\abs{\gamma(x) - \tilde\gamma(x)})
		+
		\norm{w - \tilde w}_{\Cb^1(\R^d)}
		\\&
		\le
		\zeta\parens*{
			\norm{w - \tilde w}_{\Cb(\R^d)} / \nu
		}
		+
		\norm{w - \tilde w}_{\Cb^1(\R^d)}
		.
	\end{align*}
	Thus,
	\begin{align*}
		I_2 &=
		\int_B \psi(x, \gamma(x))
		\frac{
			\abs{
				\partial_d w(x,\gamma(x))
				-
				\partial_d \tilde w(x,\tilde\gamma(x))
			}
		}{
			\partial_d w(x,\gamma(x))
			\partial_d \tilde w(x,\tilde\gamma(x))
		}
		\d\lambda^{d-1}(x)
		\\&
		\le
		\int_B \abs{\psi(x, \gamma(x))}
		\d\lambda^{d-1}(x)
		\frac{
			\zeta\parens*{
				\norm{w - \tilde w}_{\Cb(\R^d)} / \nu
			}
			+
			\norm{w - \tilde w}_{\Cb^1(\R^d)}
		}{\nu^2/2}
		\\
		&\le
		\int_B \int_a^b \abs{\partial_d \psi(x,t)} \d t
		\d\lambda^{d-1}(x)
		\frac{
			\zeta\parens*{
				\norm{w - \tilde w}_{\Cb(\R^d)} / \nu
			}
			+
			\norm{w - \tilde w}_{\Cb^1(\R^d)}
		}{\nu^2/2}
		\\
		&\le
		(b-a)^{q'} \lambda^{d-1}(B)^{q'}
		\norm{\psi}_{W^{1,q}(B \times I)}
		\frac{
			\zeta\parens*{
				\norm{w - \tilde w}_{\Cb(\R^d)} / \nu
			}
			+
			\norm{w - \tilde w}_{\Cb^1(\R^d)}
		}{\nu^2/2}
		.
	\end{align*}
	Since both estimates
	are uniform w.r.t.\ the $W^{1,q}(B \times I)$-norm,
	the density of $C_c^1(B \times I)$ in $W_0^{1,q}(B \times I)$
	implies
	that these estimates also hold for all $\psi \in W_0^{1,q}(B \times I)$.
	Note that the continuity of $J$ and $\tilde J$
	w.r.t.\ $\psi \in W_0^{1,q}(B \times I)$ follows
	from \cref{lem:trace}.

	To summarize,
	for every $p \in \set{w = 0} \cap \Omega_{\eta_0}$ with $\partial_d w(p) > 0$,
	we find an open neighborhood $W_p := B \times I$ of $p$
	and a function $\sigma_p \colon [0,\infty) \to [0,\infty]$
	with $\sigma_p(t) \to 0$ as $t \to 0$
	such that
	\begin{equation*}
		\abs*{
			\int_{\set{w = 0}} \frac{\psi}{\abs{\nabla w}} \d\HH^{d-1}
			-
			\int_{\set{\tilde w = 0}} \frac{\psi}{\abs{\nabla \tilde w}} \d\HH^{d-1}
		}
		\le
		\norm{\psi}_{W^{1,q}(W_p)}
		\sigma_p( \norm{ w - \tilde w }_{\Cb^1(\R^d) } )
	\end{equation*}
	for all
	$\tilde w \in C^1(\bar\Omega)$,
	$\psi \in W_0^{1,q}(W_p)$.
	Since we assumed $\nabla w \ne 0$ on $\set{w = 0} \cap \overline{\Omega_{\eta_0}}$,
	the same argument can be applied to every point
	$p \in \set{w = 0} \cap \overline{\Omega_\eta}$
	(by using rotated coordinate systems).

	The set $\set{w = 0} \cap \overline{\Omega_\eta}$ is compact,
	thus we can find finitely many
	$p_1, \ldots, p_n \in \set{w = 0} \cap \overline{\Omega_\eta}$
	such that the associated neighborhoods
	$W_i := W_{p_i} \subset \Omega_{\eta_0}$ cover $\set{w = 0} \cap \overline{\Omega_\eta}$.

	There is $\tau>0$ such that $|w| \ge 2\tau$ on the compact set $\overline{\Omega_\eta} \setminus \left( \bigcup_{i=1}^n W_i\right)$.
	We further set
	$W_{n + 1} := \Omega_{\eta_0} \cap \set{|w| > \tau}$.
	Note that the open sets $W_1,\ldots,W_{n+1}$ cover $\overline{\Omega_\eta}$.
	Thus, there exist functions
	$\varphi_j \in C_c^1(W_j; [0,1])$, $j = 1,\ldots,n+1$
	such that
	$\sum_{j = 1}^{n+1} \varphi_j = 1$ on $\overline{\Omega_\eta}$.
	Let now $\psi \in W_0^{1,q}(\Omega_{\eta_0})$ be given, which is extended by $0$.
	Consequently,
	\begin{align*}
		I
		&:=
		\abs*{
			\int_{\set{w = 0}} \frac{\psi}{\abs{\nabla w}} \d\HH^{d-1}
			-
			\int_{\set{\tilde w = 0}} \frac{\psi}{\abs{\nabla\tilde w}} \d\HH^{d-1}
		}
		\\
		&\le
		\sum_{j = 1}^{n + 1}
		\abs*{
			\int_{\set{w = 0}} \frac{\varphi_j \psi}{\abs{\nabla w}} \d\HH^{d-1}
			-
			\int_{\set{\tilde w = 0}} \frac{\varphi_j \psi}{\abs{\nabla\tilde w}} \d\HH^{d-1}
		}
		.
	\end{align*}
	On the support of $\varphi_{n+1}$,
	we have $\abs{w} > \tau$,
	so that the term with $j = n + 1$ vanishes if
	$\norm{w - \tilde w}_{\Cb(\R^d)}$ is smaller than $\tau$.
	The terms with $j \in \set{1,\ldots,n}$ can be handled by the local construction from above
	applied to $\varphi_j \psi \in W_0^{1,q}(W_j)$.
	This shows
	\begin{align*}
		I
		&\le
		\sum_{j = 1}^{n}
		\norm{\varphi_j \psi}_{W^{1,q}(W_p)}
		\sigma_{p_n}( \norm{ w - \tilde w }_{\Cb^1(\R^d) } )
		\le
		C
		\norm{\psi}_{W^{1,q}(\Omega_{\eta_0})}
		\sum_{j = 1}^{n}
		\sigma_{p_n}( \norm{ w - \tilde w }_{\Cb^1(\R^d) } )
		\\&
		\le
		\varepsilon
		\norm{\psi}_{W^{1,q}(\Omega_{\eta_0})}
	\end{align*}
	for
	$\norm{w - \tilde w}_{\Cb(\R^d)}$ small enough.
	This shows the claim.
\end{proof}

It remains to prove an estimate similar to the previous lemma
but only on the set $\bar\Omega$.
This will be established by the combination of \cref{lem:cty_bdry,lem:cty_for_extensions}.

\begin{theorem}
	\label{thm:cty_derivative2}
	Let $w \in C^1(\bar\Omega)$ with $\nabla w \ne 0$ on $\set{w = 0}$ be given
	such that \eqref{eq:asm_bdry} holds.
	Then, for every $\varepsilon > 0$, there exist $\delta > 0$
	such that
	\begin{equation*}
		\abs*{
			\int_{\set{w = 0}} \frac{\psi}{\abs{\nabla w}} \d\HH^{d-1}
			-
			\int_{\set{\tilde w = 0}} \frac{\psi}{\abs{\nabla\tilde w}} \d\HH^{d-1}
		}
		\le
		\varepsilon
		\norm{\psi}_{W^{1,q}(\Omega)}
		\qquad
		\forall \psi \in W^{1,q}(\Omega)
	\end{equation*}
	holds
	for every $\tilde w \in C^1(\bar\Omega)$
	with $\norm{w - \tilde w}_{C^1(\bar\Omega)} \le \delta$.
	Thus,
	the function $\sign' \colon C^1(\bar\Omega) \to \LL(C^1(\bar\Omega), W^{1,q}(\Omega)\dualspace)$
	is continuous at $w$.
\end{theorem}
\begin{proof}
	Let $\varepsilon > 0$ be given,
	choose $\eta_0 > 0$ as in \cref{lem:cty_for_extensions}
	and $\eta \in (0, \eta_0)$ according to \cref{lem:cty_bdry}.
	Next, we choose functions
	$\varphi_1 \in C_c^\infty(\Omega; [0,1])$
	and
	$\varphi_2 \in C_c^\infty(\partial_\eta \Omega; [0,1])$
	as in the proof of \cref{lem:differentiability_signum},
	i.e.,
	such that
	\begin{equation*}
		\varphi_1 = 1 \text{ on } \Omega \setminus \partial_{\eta/2}\Omega
		\quad\text{and}\quad
		\varphi_2 = 1 \text{ on } \Omega \cap \set{ \varphi_1 < 1 }
		.
	\end{equation*}
	This construction implies $1-\varphi_1 \le \varphi_2$ on $\Omega \cap \partial_\eta \Omega$.
	Finally,
	we can utilize \cref{lem:cty_for_extensions}
	with $\varepsilon$ replaced by
	\begin{equation*}
		\hat\varepsilon
		:=
		\frac{\varepsilon}{
			\norm{ \varphi_1 }_{C^1(\bar\Omega) }
			+
			\norm{\varphi_2}_{C^1(\overline{\partial_\eta \Omega})}
		}
	\end{equation*}
	to choose $\delta > 0$.

	Now, let $\tilde w \in C^1(\bar\Omega)$ with $\norm{w - \tilde w}_{C^1(\bar\Omega)} \le \delta$ be given.
	Further, let $\psi \in W^{1,q}(\Omega)$ be arbitrary.
	We have
	\begin{align*}
		I
		&:=
		\abs*{
			\int_{\set{w = 0}} \frac{\psi}{\abs{\nabla w}} \d\HH^{d-1}
			-
			\int_{\set{\tilde w = 0}} \frac{\psi}{\abs{\nabla\tilde w}} \d\HH^{d-1}
		}
		\\
		&\le
		I_1 + I_2 + I_3
		:=
		\abs*{
			\int_{\set{w = 0}} \frac{\varphi_1 \psi}{\abs{\nabla w}} \d\HH^{d-1}
			-
			\int_{\set{\tilde w = 0}} \frac{\varphi_1 \psi}{\abs{\nabla\tilde w}} \d\HH^{d-1}
		}
		\\&\qquad
		+
		\abs*{
			\int_{\set{w = 0}} \frac{(1 - \varphi_1) \psi}{\abs{\nabla w}} \d\HH^{d-1}
		}
		+
		\abs*{
			\int_{\set{\tilde w = 0}} \frac{(1 - \varphi_1) \psi}{\abs{\nabla\tilde w}} \d\HH^{d-1}
		}
		.
	\end{align*}
	Since $\varphi_1$ vanishes outside of $\Omega$,
	we can replace $w$, $\tilde w$, and $\psi$ by $E w$, $E \tilde w$, and $E_W \psi$
	respectively,
	in
	$I_1$.
	Consequently,
	\cref{lem:cty_for_extensions} provides us with the bound
	\begin{equation*}
		I_1
		\le
		\hat\varepsilon \norm{ \varphi_1 E_W \psi }_{W^{1,q}(\Omega_{\eta_0}) }
		\le
		\hat\varepsilon C \norm{ \varphi_1 }_{C^1(\bar\Omega)} \norm{ \psi }_{W^{1,q}(\Omega) }
		\le
		\varepsilon C \norm{ \psi }_{W^{1,q}(\Omega) }
		.
	\end{equation*}
	From \cref{lem:cty_bdry}, we get
	\begin{equation*}
		I_2
		\le
		\int_{\set{E w = 0} \cap \partial_\eta \Omega} \frac{\abs{E_W \psi} }{\abs{\nabla w}} \d\HH^{d-1}
		\le
		\varepsilon \norm{E_W \psi}_{W^{1,q}(\R^d)}
		\le
		\varepsilon C \norm{\psi}_{W^{1,q}(\Omega)}
		,
	\end{equation*}
	where we used that $1-\varphi$ vanishes on $\Omega \setminus \partial_\eta \Omega$ and $0 \le \varphi_1 \le 1$.
	It remains to estimate $I_3$.
	From
	$1 - \varphi_1 = 0$ on $\Omega \setminus \partial_\eta \Omega$
	and
	$1 - \varphi_1 \le \varphi_2$ on $\partial_\eta \Omega \cap \Omega$
	we get
	\begin{align*}
		I_3
		=
		\abs*{
			\int_{\set{\tilde w = 0}} \frac{(1 - \varphi_1) \psi}{\abs{\nabla\tilde w}} \d\HH^{d-1}
		}
		&\le
		\int_{\set{E \tilde w = 0} } \frac{\varphi_2 \abs{E_W \psi}}{\abs{\nabla( E \tilde w)}} \d\HH^{d-1}
		.
	\end{align*}
	Since $\varphi_2$
	is supported on $\partial_\eta \Omega$, we continue with
	\begin{align*}
		I_3
		&
		\le
		\abs*{
			\int_{\set{E \tilde w = 0} } \frac{\varphi_2 \abs{E_W \psi}}{\abs{\nabla( E \tilde w)}} \d\HH^{d-1}
			-
			\int_{\set{E w = 0} } \frac{\varphi_2 \abs{E_W \psi}}{\abs{\nabla( E w)}} \d\HH^{d-1}
		}
		+
		\int_{\set{E w = 0} } \frac{\varphi_2 \abs{E_W \psi}}{\abs{\nabla( E w)}} \d\HH^{d-1}
		\\
		&=
		\hat\varepsilon \norm{\varphi_2 \abs{E_W \psi}}_{W^{1,q}(\Omega_{\eta_0})}
		+
		\int_{\set{E w = 0} \cap \partial_\eta \Omega} \frac{\abs{E_W \psi}}{\abs{\nabla( E w)}} \d\HH^{d-1}
		\\&
		\le
		\hat\varepsilon C \norm{\varphi_2}_{C^1(\overline{\partial_\eta \Omega})} \norm{E_W \psi}_{W^{1,q}(\Omega_{\eta_0})}
		+
		\varepsilon \norm{ E_W \psi }_{W^{1,q}(\R^d)}
		\\&
		\le
		C \varepsilon
		\norm{\psi}_{W^{1,q}(\Omega)}
		,
	\end{align*}
	where we again used \cref{lem:cty_bdry,lem:cty_for_extensions}
	and $0 \le \varphi_2 \le 1$.
	Consequently,
	\(
		I_3
		\le
		C \varepsilon \norm{\psi}_{W^{1,q}(\Omega)}
		.
	\)

	Collecting the above inequalities
	shows
	\begin{equation*}
		I
		\le
		C
		\varepsilon \norm{\psi}_{W^{1,q}(\Omega)}
		.
	\end{equation*}
	Since $\varepsilon > 0$ was arbitrary,
	this shows the claim.
\end{proof}

\begin{lemma}
	\label{lem:continuity_sign_in_L1}
	Let $w\in L^1(\Omega)$ such that $\lambda( \set{w=0} ) = 0$.
	Then the map $\sign: L^1(\Omega) \to L^1(\Omega)$ is continuous at $w$.
\end{lemma}
\begin{proof}
Let $(w_k)$ be given with $w_k \to w$ in $L^1(\Omega)$ and pointwise $\lambda$-almost everywhere.
Due to $\lambda( \set{w=0} ) = 0$, we have $\sign(w_k(x)) \to \sign(w(x))$ for $\lambda$-almost all $x\in \Omega$.
The claim follows by dominated convergence and a subsequence-subsequence argument.
\end{proof}

By combining \cref{thm:cty_derivative2}
with the weak differentiability result from \cref{lem:differentiability_signum},
we obtain the continuous differentiability.
\begin{theorem}
	\label{thm:differentiability_signum}
	We assume that $\HH^{d-1}(\partial\Omega) < \infty$.
	Let $w \in C^1(\bar\Omega)$ with $\nabla w \ne 0$ on $\set{w = 0}$ be given
	such that \eqref{eq:asm_bdry} holds.
	Then, for every $q > 1$,
	the function
	$\sign \colon C^1(\bar\Omega) \to W^{1,q}(\Omega)\dualspace$
	is
	Fréchet differentiable at $w$
	with derivative
	$\sign'(w) \in \LL( C^1(\bar\Omega) , W^{1,q}(\Omega)\dualspace )$.
\end{theorem}
\begin{proof}
	From \cref{lem:differentiability_signum}
	we get that
	for fixed $\psi \in C(\bar\Omega) \cap W^{1,q}(\Omega)$,
	the scalar-valued function
	$v \mapsto \dual{\sign(v)}{\psi}_{W^{1,q}(\Omega)}$
	is Gâteaux differentiable
	at points $v \in C^1(\bar\Omega)$
	satisfying
	$\nabla v \ne 0$ on $\set{v = 0}$
	and
	$\HH^{d-1}(\set{v = 0} \cap \partial\Omega) = 0$.
	Both conditions are satisfied at $v = w$
	and,
	due to \cref{lem:stability_of_reg_condition}, the first condition is stable w.r.t.\ $C^1(\bar\Omega)$.
	Next, we fix a small $z \in C^1(\bar\Omega)$,
	such that $\nabla (w + t z) \ne 0$ on $\set{w + t z = 0}$ for all $t \in [0,1]$.
	Then, the family
	$\seq{ \set{w + t z = 0} \cap \set{w \ne 0} \cap \partial\Omega }_{t \in [0,1]}$
	consists of disjoint and measurable subsets of $\partial\Omega$.
	Due to $\HH^{d-1}(\partial\Omega) < \infty$ and $\HH^{d-1}(\set{w = 0} \cap \partial\Omega) = 0$,
	the condition
	$\HH^{d-1}(\set{w + t z = 0} \cap \partial\Omega) = 0$
	is violated for at most countably many values of $t \in [0,1]$.
	Consequently, the function
	$\Phi \colon [0,1] \to \R$,
	\begin{equation*}
		\Phi(t)
		:=
		\dual{\sign(w + t z)}{\psi}_{W^{1,q}(\Omega)}
	\end{equation*}
	is differentiable outside of a countable subset of $[0,1]$
	and  continuous on $[0,1]$ by \cref{lem:continuity_sign_in_L1}.
	Moreover, the derivative
	\begin{equation*}
		\Phi'(t)
		=
		\dual{\sign'(w + t z) z}{\psi}_{W^{1,q}(\Omega)}
	\end{equation*}
	is bounded w.r.t.\ $t \in [0,1]$, which follows from  \cref{thm:cty_derivative2}.
	In addition, $t \mapsto \Phi'(t)$ is
	measurable
	as the pointwise a.e.\@ limit of the difference quotients $(\Phi(t + \frac1n) - \Phi(t)) / n^{-1}$.
	Consequently, $\Phi'$ is Lebesgue integrable.

	Therefore, we can apply the fundamental theorem of calculus in the formulation of
	\cite[Theorem~1]{Koliha2006},
	i.e.,
	\begin{align*}
		\dual{ \sign(w + z ) - \sign(w)}{\psi}_{W^{1,q}(\Omega)}
		&=
		\Phi(1) - \Phi(0)
		=
		\int_0^1 \Phi'(t) \d t
		\\&
		=
		\int_0^1
		\dual{ \sign'(w + t z) z}{\psi}_{W^{1,q}(\Omega)}
		\d t
		.
	\end{align*}
	For a given $\varepsilon > 0$,
	we choose $\delta \in (0,\rho)$ according to
	the continuity of the derivative from \cref{thm:cty_derivative2}.
	Thus, $\norm{z}_{C^1(\bar\Omega)} \le \delta$ implies
	\begin{align*}
		\MoveEqLeft
		\dual{ \sign(w + z ) - \sign(w) - \sign'(w) z}{\psi}_{W^{1,q}(\Omega)}
		\\
		&=
		\int_0^1 \dual{ \sign'(w + t z) z}{\psi}_{W^{1,q}(\Omega)} \d t
		-
		\dual{ \sign'(w) z}{\psi}_{W^{1,q}(\Omega)}
		\\
		&=
		\int_0^1
		\dual*{
			\sign'(w + t z) z - \sign'(w) z
		}{\psi}_{W^{1,q}(\Omega)}
		\d t
		\\
		&=
		2
		\int_0^1
		\parens*{
			\int_{\set{w + t z} = 0} \frac{z \psi}{\abs{\nabla (w + t z)}} \d\HH^{d-1}
			-
			\int_{\set{w = 0}} \frac{z \psi}{\abs{\nabla w}} \d\HH^{d-1}
		}
		\d t
		\\
		&\le
		2
		\varepsilon \norm{z \psi}_{W^{1,q}(\Omega)}
		\le
		C \varepsilon \norm{z}_{C^1(\bar\Omega)} \norm{\psi}_{W^{1,q}(\Omega)}
	\end{align*}
	for all $\psi \in C(\bar\Omega) \cap W^{1,q}(\Omega)$.
	The estimate is uniform w.r.t.\ the $W^{1,q}(\Omega)$-norm
	and
	$C(\bar\Omega) \cap W^{1,q}(\Omega)$
	is dense in
	$W^{1,q}(\Omega)$,
	since $\Omega$ is assumed to be an $(1,q)$-extension domain.
	Consequently,
	the same estimate holds for all $\psi \in W^{1,q}(\Omega)$.
	This finishes the proof.
\end{proof}

We give to counterexamples which shed some light on the role of
the condition \eqref{eq:asm_bdry}.
\begin{example}[Failure of differentiability \texorpdfstring{in absence of \eqref{eq:asm_bdry}}{}]
	\label{ex:not_differentiable}
	Let us consider $\Omega = (0,1)$ and
	let $w \in C^1(\bar\Omega)$ be given by $w(x) = x$.
	It is clear that $\nabla w \ne 0$ on $\set{w = 0}$,
	but \eqref{eq:asm_bdry} is violated, since
	$\HH^0(\set{w = 0} \cap \partial\Omega) = \HH^0(\set{w = 0}) = 1$.
	We fix the function $\psi \equiv 1 \in W^{1,q}(\Omega)$.
	Further, we consider the direction $z \equiv -1 \in C^1(\bar\Omega)$.
	Now, it is easy to check that
	\begin{equation*}
		\dual{\sign(w + t z)}{\psi}_{W^{1,q}(\Omega)}
		=
		\HH^1(\set{w + t z > 0}) - \HH^1(\set{w + t z < 0})
		=
		\begin{cases}
			1 & \text{if } t < 0, \\
			1 - 2 t & \text{if } t \in [0,1).
		\end{cases}
	\end{equation*}
	This shows that $\sign$ fails to be differentiable at $w$.
	It is clear that this idea can be generalized to dimensions $d > 1$.
\end{example}
\begin{example}[\texorpdfstring{Property \eqref{eq:asm_bdry}}{Boundary assumption} is not stable in \texorpdfstring{$C^1(\bar\Omega)$}{C1} for \texorpdfstring{$d>1$}{d > 1}]
Let us define the continuously differentiable function $w: \R^2 \to \R$ by
\[
 w(x_1,x_2) = x_2 - \max(0,x_1)^2.
\]
Let $\Omega = (0,1)^2$.
Then $\nabla w \ne 0$ on $\bar\Omega$ and $\set{w = 0} \cap \partial\Omega = \set0$.
Hence, the condition \eqref{eq:asm_bdry} is satisfied.
For $\epsilon>0$ set $w_\epsilon(x_1,x_2):= w( x_1 - \epsilon, x_2)$.
Then $\norm{ w_\epsilon - w }_{C^1(\bar\Omega)} \to 0$ for $\epsilon \to 0$.
In addition,  $\set{w_\epsilon = 0} \cap \partial\Omega = [0,\epsilon] \times \set0$,
i.e, $w_\epsilon$ does not satisfy \eqref{eq:asm_bdry} for $\epsilon>0$.
\end{example}

Finally, we show that condition \eqref{eq:asm_bdry}
is stable under a further assumption on $w$.
\begin{proposition}
	\label{prop:stability_3_4}
	We assume that $\partial\Omega$ is of class $C^1$.
	Let $w \in C^1(\bar\Omega)$ be given such that
	$\nabla w(x)$ and
	the unit
	outer normal vector $\nu(x)$
	are not colinear for all $x \in \set{w = 0} \cap \partial\Omega$.
	Then, \eqref{eq:asm_bdry} holds in a small ball (in $C^1(\bar\Omega)$) around $w$.
\end{proposition}
\begin{proof}
	The assumption on $\nabla w$
	holds if and only if
	the matrix
	$J(x) := (\nabla w(x), \nu(x)) \in \R^{d \times 2}$
	has full column rank for all $x \in \set{w = 0} \cap \partial\Omega$.
	This, in turn,
	is equivalent to
	$\det( J(x)^\top J(x) ) > 0$ for all $x \in \set{w = 0} \cap \partial\Omega$.
	By continuity and compactness,
	this yields the existence of $\tau > 0$
	such that
	$\det( J(x)^\top J(x) ) \ge \tau$ for all $x \in \set{w = 0} \cap \partial\Omega$.
	Since $J$ depends continuously on $w \in C^1(\bar\Omega)$,
	it is clear that the assumption on $w$ is stable in the space $C^1(\bar\Omega)$.
	Consequently, it remains to show that \eqref{eq:asm_bdry}
	holds for $w$.
	Let $x_0 \in \set{w = 0} \cap \partial\Omega$ be given.
	Locally around $x_0$, $\partial\Omega$ can be written as $\set{\varphi = 0}$
	for some $\varphi \in C^1(\R^d)$ with $\nabla\varphi(x_0) = \nu(x_0)$.
	Consequently,
	$\set{w = 0} \cap \partial\Omega$
	coincides with $\set{ F = 0 }$ locally, where $F \in C^1(\R^d, \R^2)$
	consists of the components $E w$ and $\varphi$.
	By assumption, the Jacobian $F'(x_0) = [\nabla w(x_0), \nu(x_0)]^\top$ has rank $2$.
	Consequently,
	$\set{w = 0} \cap \partial\Omega \cap U_r(x_0) = \set{ F = 0 } \cap U_r(x_0)$
	is a surface of codimension $2$ for some $r > 0$ (depending on $x_0$).
	Hence, its measure w.r.t.\ $\HH^{d-1}$ is zero.
	By compactness of $\set{w = 0} \cap \partial\Omega$
	and additivity of $\HH^{d-1}$, this shows the claim.
\end{proof}
By combining \cref{thm:differentiability_signum}
with \cref{prop:stability_3_4}
we get a nice result.
\begin{corollary}
	\label{cor:some_open_set}
	Let $\partial\Omega$ be of class $C^1$.
	Then, the set
	\begin{equation*}
		\set{
			w \in C^1(\bar\Omega)
			\given
			\nabla w \ne 0 \text{ on } \set{w = 0}
			,\quad
			\text{$\nabla w$ and $\nu$ are not colinear on $\set{w = 0} \cap \partial\Omega$}
		}
	\end{equation*}
	is open in $C^1(\bar\Omega)$
	and $\sign \colon C^1(\bar\Omega) \to W^{1,q}(\Omega)^\star$
	is continuously differentiable on this set with derivative $\sign'$.
\end{corollary}

\section{Application to bang-bang optimal control}
\label{sec:optimal_control}
In this section, we show that our results from \cref{sec:cty_derivative_signum}
enable us to prove local superlinear convergence
for Newton's method applied to the optimality system of optimal control problems
for which the solution is of bang-bang type.
In \cref{sec:linear_quadratic}
we discuss the case of a linear state equation
and in \cref{subsec:semilinear}
we consider a semilinear elliptic equation.
In both situations,
we also suggest a heuristic for the globalization of Newton's method.
For simplicity, we only discuss the situation that the control bounds are given by $\pm 1$.
The general situation is consequently achieved by a simple scaling argument.

\subsection{Linear-quadratic optimization problem}
\label{sec:linear_quadratic}
We want to use Newton's method to solve the following optimization problem:
\[
 \min_{u \in L^2(\Omega)} \frac12 \|Su - y_d\|_{L^2(\Omega)}^2
\]
subject to $ |u|\le 1 $ almost everywhere on $\Omega$. Here, $S$ is a linear and continuous operator on $L^2(\Omega)$ and $y_d\in L^2(\Omega)$ is given.
With standard arguments one can verify the existence of a solution.
Moreover, if $S$ is injective, the solution is unique.

A feasible point $\bar u$ is a solution if and only if
\[
 \bar u(x) \in \Sign( -[S^*( S\bar u - y_d)](x) )
 \qquad \text{for a.a.\ } x \in \Omega.
\]
Introducing the dual variable $\bar\xi := y_d- S \bar u \in L^2(\Omega)$, we see that $\bar u \in \Sign (S^*\bar \xi)$ if and only if
\[
 \bar \xi +S \Sign(S^* \bar \xi) -y_d \ni 0.
\]
If $\bar\xi$ is such that $\Sign( S^* \bar \xi)$ is single-valued a.e.,
i.e., if the Lebesgue measure of $\set{ S^* \bar\xi = 0}$ is zero, then the inclusion becomes
\begin{equation}
 \label{eq:equation}
 \bar \xi +S \sign(S^* \bar \xi) -y_d = 0.
\end{equation}
This is the equation we want to solve using Newton's method.

\begin{assumption}\label{ass_nice_solution}
 Let $\bar\xi$ be a solution of \eqref{eq:equation} such that $w:= S^* \bar \xi$ satisfies
 $w \in C^1(\bar\Omega)$ with $\nabla w \ne 0$ on $\set{w = 0}$ and the condition \eqref{eq:asm_bdry}.
\end{assumption}

\begin{assumption}\label{ass_S}
 We assume that there exist $p\ge 2$ and $q > 1$ such that
 $\Omega$ satisfies \cref{asm:standing}
 and
 $S \in \LL(L^2(\Omega))$ and its adjoint $S^* \in \LL(L^2(\Omega))$ satisfy
 \begin{enumerate}
  \item $S \in \LL( W^{1,q}(\Omega)\dualspace, L^p(\Omega))$
  is compact,
  \item $S^* \in \LL( L^p(\Omega), C^1(\bar\Omega))$.
 \end{enumerate}
\end{assumption}

This is the case if $S^* \in \LL(L^p(\Omega), W^{2,p}(\Omega))$ for $p>d$ and $S^* \in \LL(L^{p'}(\Omega), W^{2,p'}(\Omega))$, so that $W^{2,p}(\Omega) \hookrightarrow C^1(\bar\Omega)$
and $W^{2,p'}(\Omega)\hookrightarrow W^{1,q}(\Omega)$.
If $S$ is the solution operator of Poisson's equation
on a sufficiently regular domain $\Omega \subset \R^d$, $d \ge 2$,
we can choose any $p \in (d,\infty)$ and $q = p'$.

Let us define the map $F \colon L^p(\Omega) \to L^p(\Omega) $
by
\[
 F( \xi):= \xi + S\sign(S^*  \xi) -y_d .
\]
Note that $F(\bar\xi) = 0$.
In addition, the term $\sign'( S^* \xi)$ is well-defined in a neighborhood of $\bar\xi$.
The following map serves as a candidate of the derivative of $F$:
\[
 G(\xi) := \id + S \sign'(S^*\xi) S^*.
\]

\begin{lemma}
 \label{lem:lemma}
Let $\bar \xi \in L^p(\Omega)$ satisfy \cref{ass_nice_solution}
and let \cref{ass_S} be satisfied.
Then there is an open neighborhood $O\subset L^p(\Omega)$ of $\bar \xi$  such that
\begin{enumerate}
 \item \label{it_F_continuous} $F: O  \to L^p(\Omega)$ is continuous,
 \item \label{it_F_Frechet} $F$ is Fréchet differentiable  at $\bar \xi$ with $F'(\bar \xi) = G(\xi)$,
 \item \label{it_G_continuous} $G: O \to \LL(L^p(\Omega))$ is continuous at $\bar \xi$,
 \item \label{it_G_invertible} there is $M>0$ such that $\norm{ G(\xi)^{-1}}_{\LL(L^p(\Omega))} \le M$ for all $\xi\in O$.
\end{enumerate}
\end{lemma}
\begin{proof}
The proof relies on \cref{ass_S} and the properties of $\sign$ and $\sign'$ collected in \cref{sec:cty_derivative_signum}.
Due to \cref{lem:stability_of_reg_condition}, there is $\delta>0$ such that $\nabla \tilde w\ne0$ on $\set{\tilde w=0}$ for all $\tilde w$ with $\|S^*\bar\xi - \tilde w\|_{C^1(\bar\Omega)}<\delta$.
Set $O := S^{-*}( U_\delta(S^*\bar\xi))$.

Now, \ref{it_F_continuous} follows from \cref{lem:continuity_sign_in_L1}, \ref{it_F_Frechet} is a consequence of \cref{thm:differentiability_signum}.
Property \ref{it_G_continuous} is implied by \cref{thm:cty_derivative2}.

It remains to prove \ref{it_G_invertible}.
 The operator $G(\bar \xi) \in \LL(L^p(\Omega))$ is a compact perturbation of the identity.
 Consequently, it is a Fredholm operator with index $0$.
 Further, for every $\xi \in L^p(\Omega) \setminus \set{0}$, we have
 $\dual{\xi}{G(\bar\xi)\xi}_{L^p(\Omega)} \ge \int_\Omega\abs{\xi}^2 \d\lambda > 0$,
 which shows that $G(\bar\xi)$ is injective.
 By the Fredholm alternative theorem, $G(\bar\xi)$ is continuously invertible.

Since $\xi \mapsto G(\xi)$ is continuous by \ref{it_G_continuous},
we can make $O$ smaller to ensure that the inverses $G(\xi)^{-1}$ exist and are uniformly bounded for all $\xi \in O$.
\end{proof}
The assertions from this lemma give rise to superlinear convergence of Newton's method in $L^p(\Omega)$,
which can be proven by classical arguments.
\begin{theorem}
 \label{thm:quadratic_newton}
 Let $\bar \xi \in L^p(\Omega)$ satisfy \cref{ass_nice_solution}
 and let \cref{ass_S} be satisfied.
 Then there is an open neighborhood $O\subset L^p(\Omega)$ of $\bar \xi$  such that
 for each $\xi_0 \in O$, the iterates
 \begin{equation*}
  \xi_{k+1} := \xi_k - G(\xi_k)^{-1} F(\xi_k)
 \end{equation*}
 stay in $O$ and converge superlinearly in $L^p(\Omega)$ towards $\bar\xi$.
\end{theorem}

In the numerical computations,
it is useful to solve the Newton system only approximately
via the CQ method
as in \cite{DWachsmuth2025}.
Towards a possible globalization,
we define the map $\Phi: L^2(\Omega) \to \R$ by
\[
 \Phi( \xi ):= \frac12\norm{ \xi-y_d }_{L^2(\Omega)}^2 + \norm{ S^* \xi}_{ L^1(\Omega) }.
\]
If $\xi$ is such that the measure of $\set{S^*\xi=0}$ is zero, then $\Phi$ is Gâteaux differentiable at $\xi$ with gradient
\[
 \nabla \Phi(\xi) := \xi +S \sign(S^* \xi) -y_d.
\]
Hence, \eqref{eq:equation} is equivalent to $ \nabla \Phi(\bar\xi) =0$.
The function $\Phi$ can be used as a merit function to globalize the inexact Newton method to solve \eqref{eq:equation},
similar to the work in \cite{DWachsmuth2025}.
The resulting method performed well in numerical computations, see \cref{sec:numerics}.
A convergence analysis of such a globalized procedure is still open.
In particular, it is unclear how to deal with points $\xi_k$, where $\Phi$ is not differentiable.

\begin{remark}[Fixed-point method]
\label{rem:fixed_point}
In order to solve  bang-bang control problems, several publications used a fixed-point method, see, e.g., \cite{DeckelnickHinze2012,Fuica2024,FuicaJork2025}.
With the results of this paper, we can explain, why a fixed-point method might converge.
Let us replace the control constraints $\abs{u} \le 1$ with $\abs{u} \le u_b$ for $u_b \ge 0$.
Then, the analogue of the optimality condition \eqref{eq:equation} is
$\bar\xi = y_d - u_b S \sign(S^* \bar\xi)$.
We could try to solve this equation by the fixed point method
$\xi_{k+1} := y_d - u_b S \sign(S^* \xi_k)$.
The derivative of the iteration map $\Psi ( \xi ) := y_d - u_b S \sign(S^* \xi)$ at the reference point
$\xi$ is given by
\[
 \Psi'( \bar \xi) \delta\xi= - u_b S\sign'( S^*\bar \xi) S^*.
\]
Let $w:= S^* \bar \xi$.
Under \cref{ass_nice_solution,ass_S}, the fixed point method is well-posed in $L^p(\Omega)$.
In addition, we can estimate
\[
 \| \Psi'( \bar \xi ) \|_{\LL(L^p(\Omega))} \le u_b \|S\|_{\LL( W^{1,q}(\Omega)\dualspace, L^p(\Omega))} \|\sign'(w)\|_{\LL(C^1(\bar\Omega),W^{1,q}(\Omega)\dualspace)} \|S^*\|_{\LL( L^p(\Omega), C^1(\bar\Omega))}.
\]
For small values of $u_b$,
we could expect that the operator norm of $\Psi'(\bar\xi)$ is smaller than $1$
and, consequently, the fixed point iteration
converges at least locally.
Indeed, note that $\bar\xi \to \xi_0 := y_d$ as $u_b \to 0$.
Consequently, if $w_0 := S^* y_d$
satisfies $\nabla w_0 \ne 0$ on $\set{w_0 = 0}$
and
$\HH^{d-1}\parens*{ \set{w_0 = 0} \cap \partial\Omega } = 0$,
the continuity of $\sign'$ from \cref{thm:cty_derivative2}
enables us to prove $\norm{\Psi'(\bar\xi)} < 1$ for $u_b$ small enough.
\end{remark}

\subsection{Semilinear control problem}
\label{subsec:semilinear}
Now, we consider the problem
\begin{align*}
	\text{Minimize}& \quad  \frac12\norm{y - y_d}_{L^2(\Omega)}^2 \\
	\text{s.t.}& \quad\mathopen{} -\Delta y + a(y) = u \text{ in } \Omega, \quad \frac{\partial y}{\partial\nu} = 0 \text{ on }\partial\Omega\\
	\text{and}& \quad\mathopen{} -1 \le u \le 1
	.
\end{align*}
We assume that $\Omega\subset \R^d$ is a bounded Lipschitz domain and $a \in C^2(\R)$ with $a' \ge 1$.
Let $\bar u$ be locally optimal in $L^2(\Omega)$. Since the feasible set of the above problem is bounded in $L^\infty(\Omega)$
it follows that $\bar u$ is locally optimal in $L^r(\Omega)$ for all $r\in (d,\infty]$.
Then by classical arguments  \cite[Chapter 4]{Troltzsch2010}, \cite{Casas2012:1},
there exists an adjoint state $\bar p$ such that
\begin{equation*}
	\bar u \in \Sign(-\bar p),
	\qquad
	-\Delta \bar p + a'(\bar y) \bar p = \bar y - y_d \quad\text{in } \Omega,
	\qquad
	\frac{\partial \bar p}{\partial\nu} = 0 \quad\text{on }\partial\Omega,
\end{equation*}
As in \cref{sec:linear_quadratic},
we convert this into an equation by assuming that the set
$\set{\bar p = 0}$
has measure $0$.
Consequently, $\Sign(-\bar p)$ is single-valued a.e.\ and
the above optimality system can be written as
\begin{equation}
	\label{eq:optimality}
	F(\bar y, \bar p) = 0
	,
\end{equation}
where
$F \colon W^{1,p}(\Omega) \times W^{2,p}_*(\Omega) \to L^p(\Omega) \times W^{1,p'}(\Omega)\dualspace$
is defined via
\begin{equation*}
	F(y,p) :=
	\begin{pmatrix}
		-\Delta p + a'(y) p - y + y_d \\
		-\Delta y + a(y) - \sign(-p)
	\end{pmatrix}
	.
\end{equation*}
Here, the exponent $p$ is assumed to satisfy $p \in (d,\infty)$, see \cref{asm:assumption} below,
and
\begin{equation*}
	W^{2,p}_*(\Omega)
	:=
	\set*{
		v \in W^{2,p}(\Omega)
		\given
		\frac{\partial v}{\partial \nu} = 0\text{ on }\partial\Omega
	}
\end{equation*}
includes the Neumann boundary condition.
Note that we already have eliminated the control variable.
We are going to solve \eqref{eq:optimality}
by Newton's method.
The candidate for the derivative of $F$ is
\begin{equation*}
	G(y,p)
	:=
	\begin{pmatrix}
		-1 + a''(y) p & -\Delta + a'(y) \\
		-\Delta + a'(y) & \sign'(-p)
	\end{pmatrix}
	.
\end{equation*}
We need some assumptions.
\begin{assumption}
	\label{asm:assumption}
	Let $(\bar y, \bar u, \bar p)$ be a solution of the optimality system
	such that $\bar p \in C^1(\bar\Omega)$ satisfies
	$\nabla\bar p \ne 0$ on $\set{\bar p = 0}$
	and \eqref{eq:asm_bdry}.
	Further, we require that $\Omega \subset \R^d$ is regular enough
	such that $-\Delta + a'(\bar y)$
	is an isomorphism
	from $W^{2,p}_*(\Omega)$ to $L^p(\Omega)$
	and from
	$W^{1,p}(\Omega)$ to $W^{1,p'}(\Omega)\dualspace$,
	where the exponent $p \in (d,\infty)$ with $p \ge 2$
	is fixed.
	Finally, we require $W^{1,p}(\Omega) \embeds C(\bar\Omega)$
	and that $\Omega$ satisfies \cref{asm:standing} with $q = p'$.
\end{assumption}
\begin{lemma}
	\label{lem:lemma_on_something}
	Let \cref{asm:assumption} be satisfied.
	Then,
	there exists an open neighborhood $O \subset W^{1,p}(\Omega) \times W^{2,p}_*(\Omega)$
	of $(\bar y, \bar p)$
	such that
	\begin{enumerate}
		\item
			$F \colon O \to L^p(\Omega) \times W^{1,p'}(\Omega)\dualspace$
			is continuous,
		\item
			$F$ is Fréchet differentiable at $(\bar y, \bar p)$ with $F'(\bar y, \bar p) = G(\bar y, \bar p)$,
		\item
			$G \colon O \to \LL( W^{1,p}(\Omega) \times W^{2,p}_*(\Omega), L^p(\Omega) \times W^{1,p'}(\Omega)\dualspace)$
			is continuous at $(\bar y, \bar p)$.
	\end{enumerate}
\end{lemma}
\begin{proof}
	From \cref{asm:assumption},
	we get the embeddings
	$W^{2,p}(\Omega) \hookrightarrow C^1(\bar\Omega)$
	and
	$W^{1,p}(\Omega) \hookrightarrow L^\infty(\Omega)$.
	The Nemytskii operator associated with $a$ is $C^2$ on $L^\infty(\Omega)$.
	Thus, we can argue as in the proof of \cref{lem:lemma}.
\end{proof}
Note that we still have to check that
$G(\bar y, \bar p)$
is continuously invertible.
It is clear that
we need a second-order condition,
since we are dealing with a nonconvex problem.
\begin{assumption}
	\label{asm:SSC}
	Let $(\bar y, \bar u, \bar p)$ be a solution of the optimality system.
	We assume that
	\begin{equation}
		\label{eq:SSC}
		\int_\Omega (1 - a''(\bar y) \bar p) \parens*{ (-\Delta + a'(\bar y))^{-1} h }^2 \d\lambda
		+
		\frac12 \int_{\set{\bar p = 0}} \abs{\nabla \bar p} h^2 \d\HH^{d-1}
		>
		0
	\end{equation}
	holds for all $h \in L^2(\HH^{d-1}|_{\set{\bar p = 0}}) \setminus \set{0}$.
\end{assumption}
For a similar control problem
(with Dirichlet boundary conditions instead of Neumann boundary conditions)
it was shown
in
\cite[Theorem~6.12]{ChristofWachsmuth2017:1}
that (under suitable assumptions)
\eqref{eq:SSC}
is a sufficient optimality condition
which is equivalent to a quadratic growth condition in $L^1(\Omega)$.
We expect that this result can be transferred to the optimal control problem at hand
using the results from \cref{sec:cty_derivative_signum}.

\begin{lemma}
	\label{lem:some_invertibility}
	Let \cref{asm:assumption,asm:SSC} be satisfied.
	Then,
	there exists an open neighborhood $O \subset W^{1,p}(\Omega) \times W^{2,p}_*(\Omega)$
	of $(\bar y, \bar p)$
	and $M > 0$
	such that
	for all $(y,p) \in O$,
	$G(y,p) \in \LL( W^{1,p}(\Omega) \times W^{2,p}_*(\Omega), L^p(\Omega) \times W^{1,p'}(\Omega)\dualspace)$
	is continuously invertible
	with inverse bounded by $M$.
\end{lemma}
\begin{proof}
	In view of the continuity of $G$ from \cref{lem:lemma_on_something},
	it is sufficient to check that
	$G(\bar y, \bar p)$ is continuously invertible.
	It is clear that the operator
	\begin{equation*}
		\begin{pmatrix}
			0 & -\Delta + a'(\bar y) \\
			-\Delta + a'(\bar y) & \sign'(-\bar p)
		\end{pmatrix}
		\in
		\LL( W^{1,p}(\Omega) \times W^{2,p}_*(\Omega), L^p(\Omega) \times W^{1,p'}(\Omega)\dualspace)
	\end{equation*}
	is continuously invertible
	and $G(\bar y, \bar p)$ is a compact perturbation of this operator.
	Consequently, $G(\bar y, \bar p)$ is a Fredholm operator of index $0$.
	By the Fredholm alternative, it is sufficient to check that $G(\bar y, \bar p)$ is injective.
	Let $z \in W^{1,p}(\Omega)$, $q \in W^{2,p}_*(\Omega)$ be given
	such that
	$G(\bar y, \bar p)(z,q) = 0 \in L^p(\Omega) \times W^{1,p'}(\Omega)\dualspace$.
	We set
	\begin{equation*}
		v
		:=
		-2 \frac{q}{\abs{\nabla \bar p}} \bigg|_{\set{\bar p = 0}}
		\in C(\set{\bar p = 0})
		.
	\end{equation*}
	We identify $v$ with the functional $\varphi \mapsto \int_{\set{\bar p = 0}} \varphi v \d\HH^{d-1}$ from $W^{1,p'}(\Omega)\dualspace$,
	which is well defined due to \cref{lem:trace}.
	From the second component of
	$G(\bar y, \bar p)(z,q) = 0$,
	we get
	$z = (-\Delta + a'(\bar y))^{-1} v$.
	Due to $p \ge 2$,
	it follows $W^{1,p}(\Omega) \hookrightarrow L^{p'}(\Omega)$
	and we can test
	the equation $G(\bar y, \bar p)(z,q) = 0$
	with $(-z, q)$.
	This gives
	\begin{align*}
		0
		&=
		\dual{(-z,q)}{G(\bar y, \bar p)(z,q)}
		=
		\int_\Omega (1 - a''(\bar y)\bar p) z^2 \d\lambda
		+
		2 \int_{\set{\bar p = 0}} \frac{q^2}{\abs{\nabla\bar p}} \d\HH^{d-1}
		\\
		&=
		\int_\Omega (1 - a''(\bar y)\bar p) \parens*{(-\Delta + a'(\bar y))^{-1} v}^2 \d\lambda
		+
		\frac12 \int_{\set{\bar p = 0}} \abs{\nabla\bar p} v^2 \d\HH^{d-1}
		.
	\end{align*}
	Due to \cref{asm:SSC}, this implies
	$v = 0$
	and, consequently, $z = 0$ and $q = 0$.
	This yields injectivity of $G(\bar y, \bar p)$ and shows the claim.
\end{proof}
As a consequence, we obtain local superlinear convergence of Newton's method.
\begin{theorem}
	\label{thm:semilinear_newton}
	Let \cref{asm:assumption,asm:SSC} be satisfied.
	Then there is an open neighborhood $O\subset W^{1,p}(\Omega) \times W^{2,p}_*(\Omega)$ of $(\bar y, \bar p)$ such that
	for all $(y_0, p_0) \in O$, the iterates
	\begin{equation*}
		(y_{k+1}, p_{k+1}) := (y_k, p_k) - G(y_k, p_k)^{-1} F(y_k, p_k)
	\end{equation*}
	stay in $O$ and converge superlinearly towards $(\bar y, \bar p)$.
\end{theorem}

In what follows,
we sketch a heuristic for a possible globalization.
We denote by $S$ the solution operator of the semilinear state equation.
We define the reduced objective $f$ via
\begin{equation*}
	f(u) := \frac12 \norm{ S(u) - y_d}_{L^2(\Omega)}^2.
\end{equation*}
From the optimality system,
we know that
the optimal control $\bar u$
can be written as $\bar u = \sign(-\bar p)$,
where $\bar p = f'(\bar u)$,
under the assumption that $\set{\bar p = 0}$ has measure zero.
This motivates the definition of
\begin{equation*}
	J(w) := f(\sign(w)) = \frac12 \norm{ S(\sign(w)) - y_d}_{L^2(\Omega)}^2.
\end{equation*}
Under the assumption that $w \in C^1(\bar\Omega)$
satisfies $\nabla w \ne 0$ on $\set{w = 0}$ and \eqref{eq:asm_bdry},
we get
\begin{equation}\label{eq:Jprime}
	J'(w)
	=
	\sign'(w) f'(\sign(w))
	=
	\sign'(w) ( w + f'(\sign(w)) )
	.
\end{equation}
Here, we used that $\sign'(w) w = 0$
in $W^{1,p'}(\Omega)\dualspace$
by the definition of $\sign'$, see \eqref{eq:derivative_signum}.
If $(\bar y, \bar u)$ is a locally optimal solution of our optimization problem,
then there are, of course, many $w$ with $\bar u = \sign(w)$
and all of these satisfy $J'(w) = 0$.
Consequently, it is not sensible to solve $J'(w) = 0$,
but, instead, we aim to solve the stronger
$0 = \bar w + f'(\sign(\bar w))$,
i.e., $\bar w = -\bar p$,
where $\bar p := f'(\bar u)$ is the optimal adjoint state.
The corresponding Newton equation
at an iterate $w_k$ (at which $\sign$ is assumed to be differentiable)
reads
\begin{equation}
	\label{eq:unsymmetrisch}
	[I + f''(\sign(w_k)) \sign'(w_k)] \delta w_k
	=
	-(w_k + f'(\sign(w_k)))
	.
\end{equation}
Note that the system matrix is not symmetric.
Similarly to the arguments leading to \cref{thm:semilinear_newton},
one can show that (under \cref{asm:assumption,asm:SSC})
this method converges superlinearly in a neighborhood of $\bar w = -\bar p$.

In order to symmetrize \eqref{eq:unsymmetrisch},
we multiply it by $\sign'(w_k)$ from the left and arrive at
\begin{equation}
	\label{eq:symmetrisch}
	[\sign'(w_k) + \sign'(w_k) f''(\sign(w_k)) \sign'(w_k)] \delta w_k
	=
	-\sign'(w_k) (w_k + f'(\sign(w_k)))
	.
\end{equation}
It is clear that \eqref{eq:symmetrisch} does not imply \eqref{eq:unsymmetrisch}.
Note that \eqref{eq:symmetrisch} can be written as
\begin{equation*}
	H(w_k) \delta w_k = -J'(w_k)
\end{equation*}
with
\begin{equation*}
	H(w)
	:=
	\sign'(w) [I + f''(\sign(w)) \sign'(w)]
\end{equation*}
and this operator can be seen as a Gauß--Newton approximation
of the
first derivative of the
expression of $J'$ derived in \eqref{eq:Jprime}
(which might, in general, fail to exist).
Indeed, we only differentiate the expression $w + f'(\sign(w))$
in $J'(w)$
and ignore the (non-existing) derivative of $\sign'(w)$,
which, at the solution $\bar w$, would be multiplied by $\bar w + f'(\sign(\bar w)) = 0$ anyway.
For globalization, we use a trust-region approach,
i.e., we (approximately) solve
the trust-region subproblem
\begin{equation*}
	\text{Minimize}
	\quad
	\frac12 \dual{\delta w_k}{H(w_k) \delta w_k} + J'(w_k) \delta w_k
	\quad\text{s.t.}\quad
	\dual{\delta w_k}{\sign'(w_k) \delta w_k} \le \Delta^2
	.
\end{equation*}
This will be done by applying the Steihaug CG method
to solve \eqref{eq:unsymmetrisch}
in the semi-definite bilinear form
\begin{equation}
	\label{eq:semi_inner_product}
	\innerprod{v_1}{v_2}_{w_k}
	:=
	\dual{v_1}{\sign'(w_k) v_2}
	.
\end{equation}
Note that the operator appearing in \eqref{eq:unsymmetrisch}
is self-adjoint w.r.t.\ this bilinear form.
By
\begin{equation*}
	\norm{v}_{w_k}
	:=
	\sqrt{\dual{v}{\sign'(w_k) v}}
\end{equation*}
we denote the associated seminorm.
The Steihaug CG is stopped when the iterates leave the trust-region
$\set{v \in C^1(\bar\Omega) \given \norm{v}_{w_k} \le \Delta_k}$
or if the residual in \eqref{eq:unsymmetrisch} is small,
i.e.,
\begin{equation*}
	\norm*{
		[I + f''(\sign(w_k)) \sign'(w_k)] \delta w_k
		+
		(w_k + f'(\sign(w_k)))
	}_{w_k}
	\le \tau_k
\end{equation*}
for given $\tau_k > 0$.
One can check that the resulting Steihaug CG is well-defined.
These ideas give rise to \cref{alg:ATRM}.
\begin{algorithm2e}
	\KwData{%
		$0 < \eta_2 \le \eta_1 < 1$,
		$0 < \gamma_1 < 1 < \gamma_2$,
		$\Delta_0 > 0$,
		$\theta_1 > 1$,
		$\theta_2 > 0$,
		$w_0 \in C^1(\bar\Omega)$
	}
	\For{$k \leftarrow 0$ \KwTo $\infty$}
	{
		Solve \eqref{eq:unsymmetrisch} for $\delta w_k$ using the Steihaug CG method
		in the bilinear form given by \eqref{eq:semi_inner_product}
		using the trust-region radius $\Delta_k$
		and absolute tolerance $\tau_k := \min( \norm{w_k - f'(\sign(w_k))}_{w_k}^{\theta_1}, \theta_2 \norm{w_k - f'(\sign(w_k))}_{w_k} )$\;
		Set
		$\operatorname{pred}_k := -\bracks*{ \frac12 \dual{\delta w_k}{H(w_k) w_k} + J'(w_k) \delta w_k }$,
		$\operatorname{ared}_k := J(w_k) - J(w_k + \delta w_k)$,
		and
		$\rho_k := \frac{\operatorname{ared}_k}{\operatorname{pred}_k}$\;
		\eIf{$\rho \ge \eta_2$}
		{
			$w_{k+1} = w_k + \delta w_k$\;
			\eIf{$\rho \ge \eta_1$}
			{
				$\Delta_{k+1} := \max(\Delta_k, \gamma_2 \norm{\delta w_k}_{w_k} )$\;
			}{
				$\Delta_{k+1} := \Delta_k$\;
			}
		}
		{
			$w_{k+1} = w_k$\;
			$\Delta_{k+1} = \gamma_1 \norm{\delta w_k}_{w_k}$\;
		}
	}
	\caption{Trust-region globalization}
	\label{alg:ATRM}
\end{algorithm2e}
In practise,
this algorithm works quite nicely, see \cref{sec:numerics}.
This is a little bit surprising, since
the subproblem (with trust-region radius $\Delta := \infty$)
only yields a solution to \eqref{eq:symmetrisch}
and we would not expect convergence to a solution of
$0 = \bar w + f'(\sign(\bar w))$,
which would require the solution of \eqref{eq:unsymmetrisch}.
Analyzing this globalization is subject to future research.

In our numerical experiments,
we used the parameters
\begin{equation*}
	\eta_1 = \frac34,
	\quad
	\eta_2 = \frac14,
	\quad
	\gamma_1 = \frac12,
	\quad
	\gamma_2 = 2,
	\quad
	\theta_1 = \frac32,
	\quad
	\theta_2 = \frac{1}{20}
	.
\end{equation*}

\section{Numerical examples}
\label{sec:numerics}
In this section,
we give some numerical results for the solution of the problem
\begin{align*}
	\text{Minimize}& \quad  \frac12\norm{y - y_d}_{L^2(\Omega)}^2 \\
	\text{s.t.}& \quad\mathopen{} -\Delta y + 10 y + \alpha y^3 = u \text{ in } \Omega, \quad \frac{\partial y}{\partial\nu} = 0 \text{ on }\partial\Omega\\
	\text{and}& \quad\mathopen{} \abs{u} \le u_b
\end{align*}
on the square $\Omega = (-1,1)^2$.
The desired state is given by
\begin{equation*}
	y_d(x_1, x_2) = 2\sin(2\pi(x_1+1)/2)\cos(2\pi x_2) + 0.8 (x_2+x_1^2) - 0.5
	.
\end{equation*}
In the linear case $\alpha = 0$,
we can apply both
globalized semismooth Newton
methods from \cref{sec:optimal_control}.
For the nonlinear case, we use $\alpha = 3$
and employ \cref{alg:ATRM} for the numerical solution.
For the discretization, we use the finite element method.
The state $y$ and the adjoint $p$ are discretized by piecewise linear functions
and we use the variational discretization for the control variable.
The algorithms are implemented in Julia 1.11.4.

The above problem becomes harder for larger values of $u_b$.
If $u_b$ is small (below approximately $57$ in the linear case and below approximately $60$ in the nonlinear case), the optimal solution seems to be of bang-bang type,
see \cref{fig:solutions}.
\begin{figure}[ht]
	\centering
	\includegraphics[width=.30\textwidth]{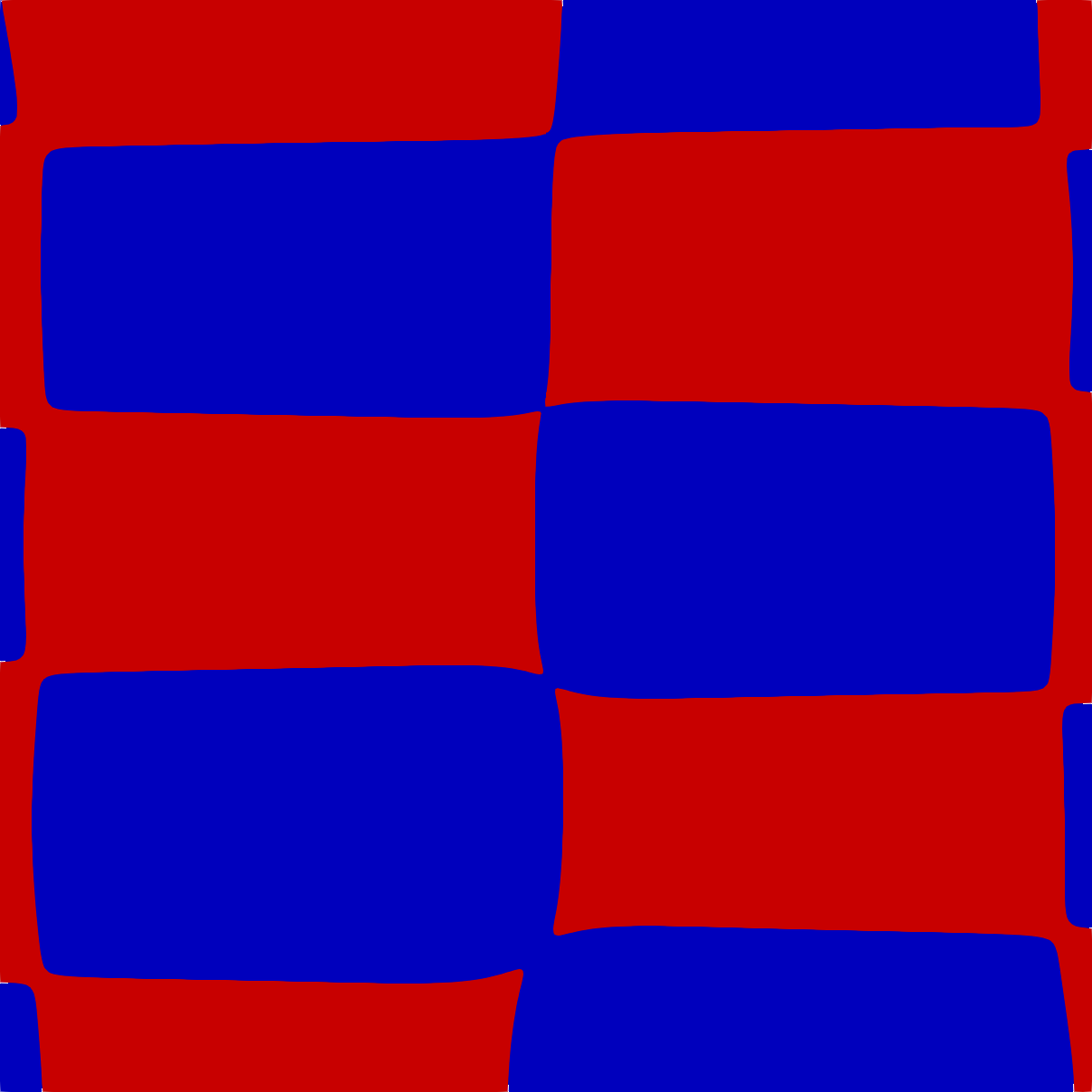}\hspace{.5cm}%
	\includegraphics[width=.30\textwidth]{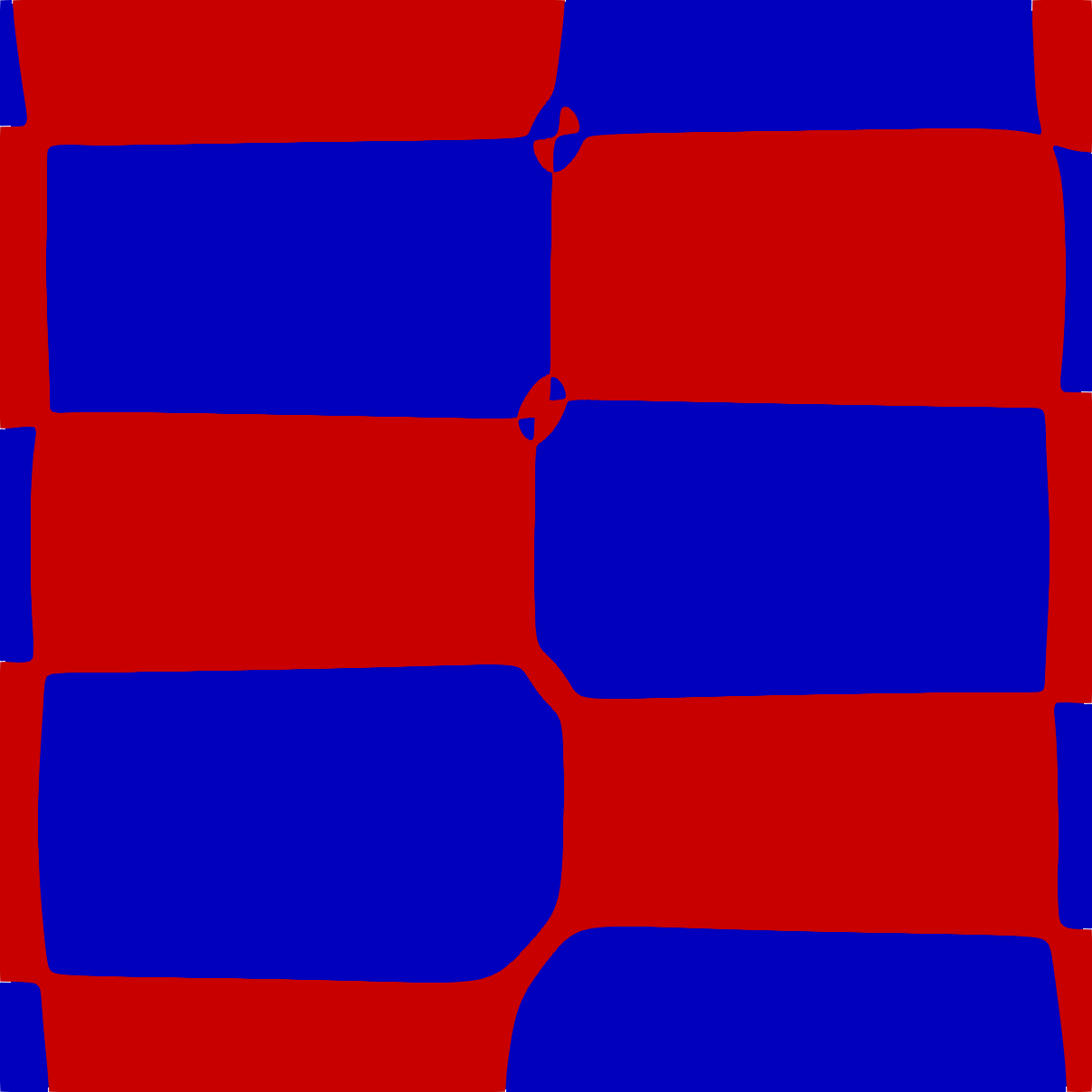}\\[.5cm]
	\includegraphics[width=.30\textwidth]{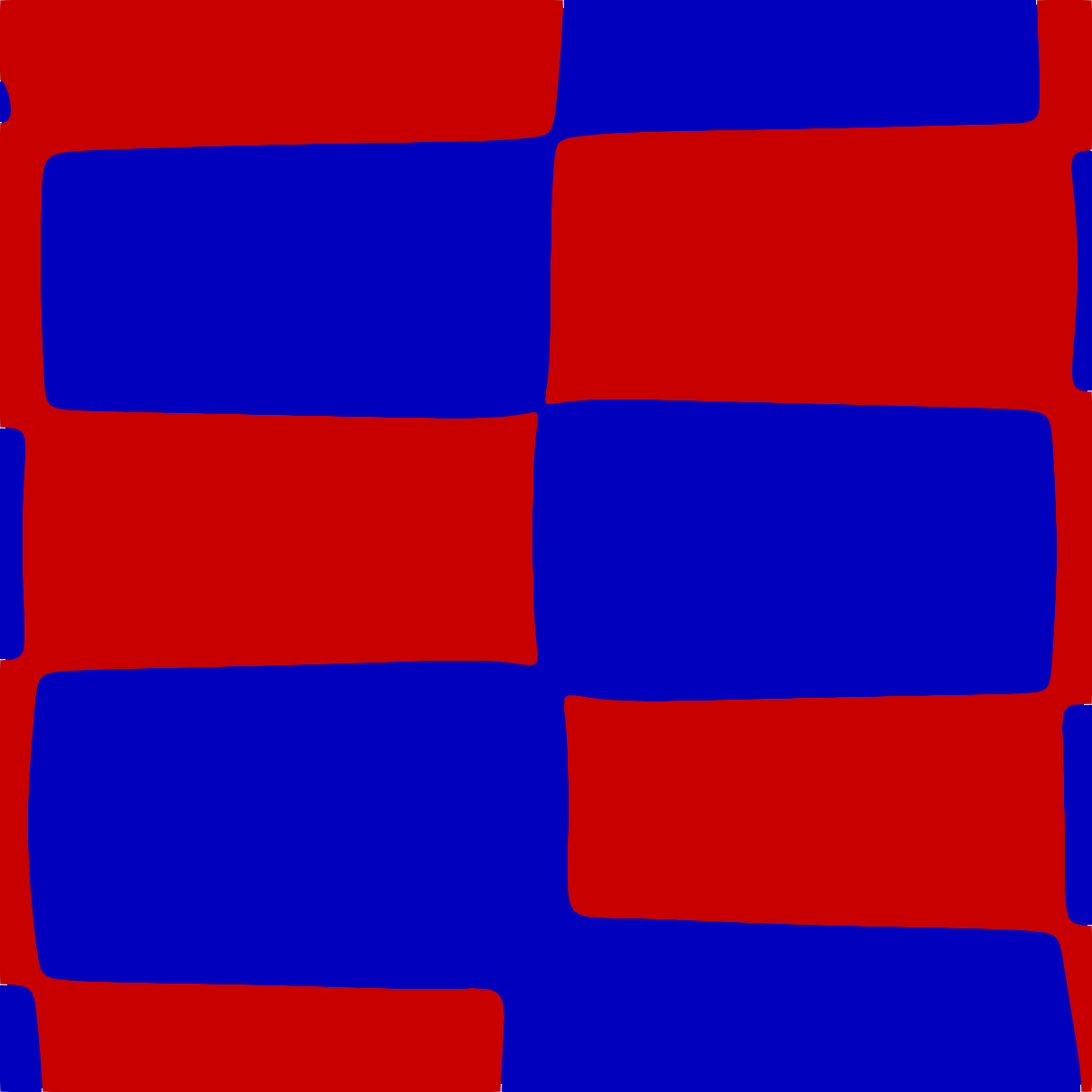}\hspace{.5cm}%
	\includegraphics[width=.30\textwidth]{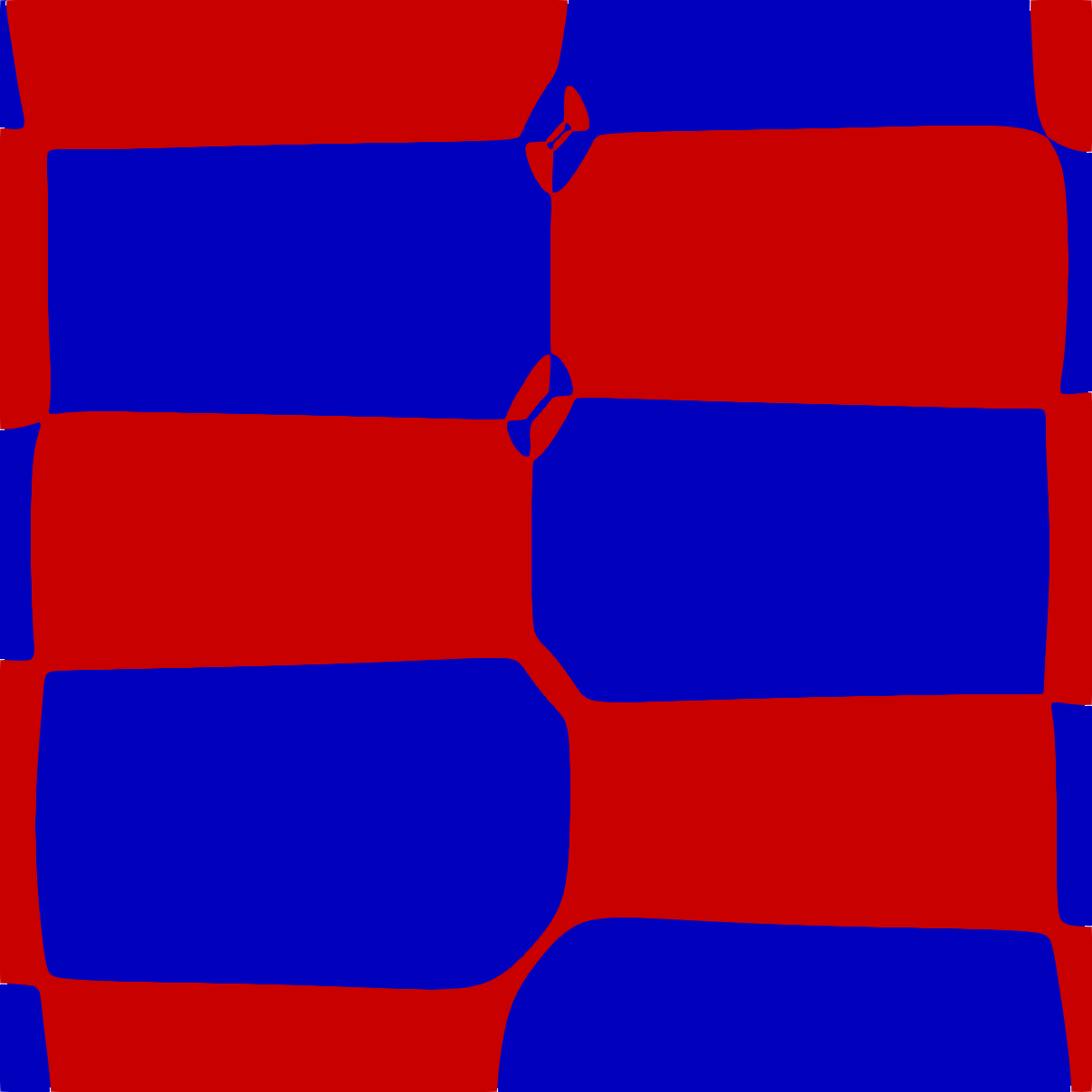}
	\caption{%
		Numerically computed optimal controls in the linear case (top row)
		with $u_b = 50$ (left) and $u_b = 57$ (right),
		and in the nonlinear case (bottom row)
		with $u_b = 50$ (left) and $u_b = 60$ (right).
		Blue and red colors indicate that the control is at the lower bound $-u_b$ and at the upper bound $u_b$, respectively.
	}
	\label{fig:solutions}
\end{figure}
In the actual computations for the linear case,
the main work is the application of the (discretized) control-to-state map $S$,
which is sped up by storing the Cholesky factorization of the stiffness matrix.
The results for $u_b = 50$ and $u_b = 57$
are given in \cref{tab:numerical_linear_50,tab:numerical_linear_57}, respectively.
Therein, ``iters'' is the number of (outer) iterations and ``solves''
is the number of applications of $S$.
The run time is given in seconds.
\begin{table}[ht]
	\centering
	\begin{tabular}{rrrrrrr}
		\toprule
		& \multicolumn{3}{c}{Alg.\ from \cref{sec:linear_quadratic}} & \multicolumn{3}{c}{Alg.\ from \cref{subsec:semilinear}}\\
		\cmidrule(r){2-4} \cmidrule(r){5-7}
		\# nodes & iters & solves & time  & iters & solves & time \\
		    81 &  15 & 494 &  0.0024 & 48 &  724 &  0.0053 \\
		   289 &  15 & 486 &  0.0068 & 57 &  825 &  0.0190 \\
		  1089 &  15 & 534 &  0.0311 & 71 &  921 &  0.0694 \\
		  4225 &  16 & 530 &  0.1876 & 66 &  967 &  0.3440 \\
		 16641 &  15 & 514 &  0.7408 & 81 & 1088 &  1.6900 \\
		 66049 &  16 & 592 &  5.3174 & 68 &  954 &  8.2823 \\
		263169 &  16 & 592 & 25.3966 & 72 &  981 & 36.8140 \\
		\bottomrule
	\end{tabular}
	\caption{Numerical results in the linear case with $u_b = 50$.}
	\label{tab:numerical_linear_50}
\end{table}
\begin{table}[ht]
	\centering
	\begin{tabular}{rrrrrrr}
		\toprule
		& \multicolumn{3}{c}{Alg.\ from \cref{sec:linear_quadratic}} & \multicolumn{3}{c}{Alg.\ from \cref{subsec:semilinear}}\\
		\cmidrule(r){2-4} \cmidrule(r){5-7}
		\# nodes & iters & solves & time  & iters & solves & time \\
		    81 &  17 &  616 &  0.0029 &  48 &  788 &  0.0049 \\
		   289 &  16 &  572 &  0.0078 &  43 &  679 &  0.0104 \\
		  1089 &  18 &  666 &  0.0370 &  55 &  795 &  0.0410 \\
		  4225 &  19 &  766 &  0.2500 &  87 & 1514 &  0.4574 \\
		 16641 &  21 &  986 &  1.4308 & 144 & 2381 &  3.5608 \\
		 66049 &  28 & 1124 & 10.0760 & 144 & 2454 & 22.6534 \\
		263169 &  28 & 1224 & 49.0543 & 159 & 2609 & 96.7201 \\
		\bottomrule
	\end{tabular}
	\caption{Numerical results in the linear case with $u_b = 57$.}
	\label{tab:numerical_linear_57}
\end{table}
It can be seen that the number of iterations and the number of applications of $S$
is almost mesh independent for $u_b = 50$.
In the case of $u_b = 57$,
the very fine structure of the solution is not resolved on coarse meshes
and therefore the iteration numbers grow slightly.

If we further increase $u_b$, our method breaks down.
Let us report about the outcome of the computations for $u_b=100$.
Although the iterates $w_k = S^* \xi_k$ still satisfy the assumptions from \cref{sec:cty_derivative_signum},
the operator norm of $\sign'(w_k)$
seems not to be bounded
uniformly w.r.t.\ the iteration number $k$.
In particular, we observe that the number of CG iterations needed to solve the Newton system
does not stay bounded in $k$.
We conjecture
that this is due to the fact that optimal controls are no longer bang-bang for large values of $u_b$.
Consequently, \cref{ass_nice_solution} is not satisfied and our convergence result \cref{thm:quadratic_newton}
does not apply.
In \cref{fig:solutions2}, we show the iterate $u_k = \sign(-S^*\xi_k)$ after $20$ iterations.
\begin{figure}[ht]
	\centering
	\includegraphics[width=.30\textwidth]{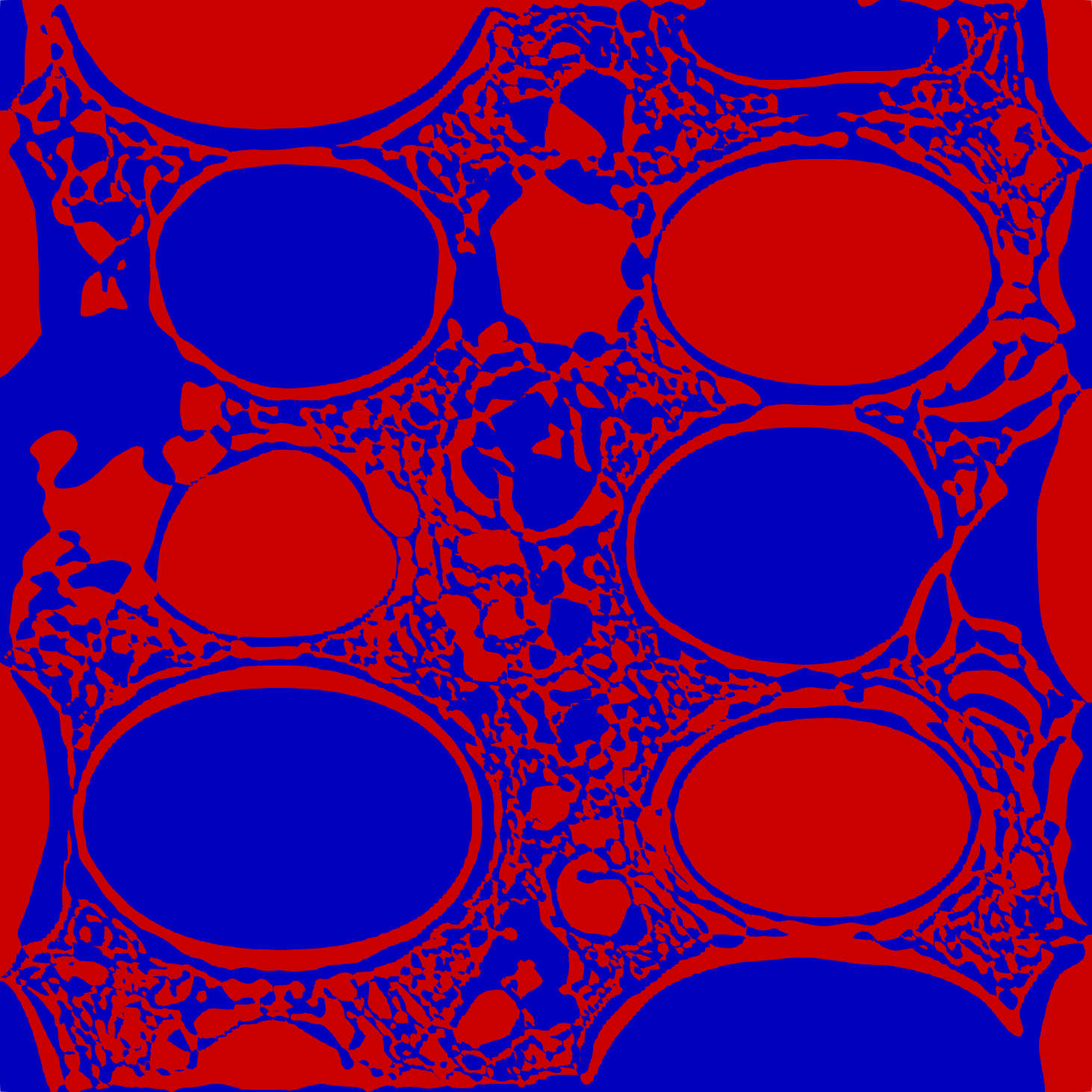}
	\caption{%
		The control $u_k$ after $20$ iterations
		in the linear case with $u_b = 100$.
		Blue and red colors indicate that the control is at the lower bound $-u_b$ and at the upper bound $u_b$, respectively.
	}
	\label{fig:solutions2}
\end{figure}

We would also like to point out that the fixed-point method from \cref{rem:fixed_point}
seems to converge for $u_b \le 2.7$ and seems to diverge for $u_b \ge 2.75$.

The results for the nonlinear case are given in \cref{tab:numerical_nonlinear}.
\begin{table}[ht]
	\centering
	\begin{tabular}{rrrrrrrrr}
		\toprule
		& \multicolumn{4}{c}{$u_b = 50$} & \multicolumn{4}{c}{$u_b = 60$}\\
		\cmidrule(r){2-5} \cmidrule(r){6-9}
		\# nodes & iters & fact. & sol. & time  & iters & fact. & sol. & time \\
		    81 &  37 &  209 &  743 &   0.0066 &  44 &  253 &  907 &   0.0096 \\
		   289 &  36 &  194 &  731 &   0.0277 &  49 &  276 &  972 &   0.0424 \\
		  1089 &  49 &  242 &  932 &   0.1696 &  49 &  264 &  989 &   0.1796 \\
		  4225 &  49 &  256 &  904 &   1.0579 & 110 &  465 & 2248 &   1.8969 \\
		 16641 &  42 &  221 &  776 &   3.7670 & 115 &  481 & 2135 &   8.8254 \\
		 66049 &  45 &  235 &  871 &  20.8764 & 154 &  597 & 3275 &  64.0454 \\
		263169 &  51 &  261 &  994 & 110.0614 & 250 &  916 & 4586 & 419.9417 \\
		\bottomrule
	\end{tabular}
	\caption{Numerical results in the nonlinear case.}
	\label{tab:numerical_nonlinear}
\end{table}
Now, we have to solve the linear systems corresponding to the operator
$(-\Delta + 10 + 3 \alpha w^2)$ for various values of $w$.
Whenever a new value of $w$ is used, the associated matrix has to be factorized
(``fact.'')
and for subsequent solves with the same matrix (``sol.''),
this factorization can be reused.
Similar to the linear problem,
we observe almost mesh-independence for $u_b = 50$
and a slight mesh dependence for $u_b = 60$,
which, again, is caused by the fine structures of the solution.

\printbibliography

\end{document}